\documentclass{article}
\usepackage[utf8]{inputenc}
\usepackage{enumerate}
\usepackage{amssymb, bm}
\usepackage{amsthm}
\usepackage{amsmath}
\usepackage{tikz-cd} 
\usepackage{tikz}
\usepackage{bbm}
 \usepackage[all]{xy}
\title{On compact K\"ahler manifold with strongly pseudo-effective tangent bundle}
\author{Xiaojun WU}
\date{\today}
\newtheorem{mythm}{Theorem}
\newtheorem{mylem}{Lemma}
\newtheorem{myprop}{Proposition}
\newtheorem{myex}{Example}
\newtheorem{mycor}{Corollary}
\newtheorem{mydef}{Definition}
\newtheorem{myrem}{Remark}
\begin{document}
\maketitle
  \def\tors{\mathrm{Tors}}
\def\cI{\mathcal{I}}
\def\Z{\mathbb{Z}}
\def\Q{\mathbb{Q}}  \def\C{\mathbb{C}}
 \def\R{\mathbb{R}}
 \def\N{\mathbb{N}}
 \def\H{\mathbb{H}}
  \def\P{\mathbb{P}}
 \def\rC{\mathcal{C}}
  \def\nd{\mathrm{nd}}
  \def\d{\partial}
 \def\dbar{{\overline{\partial}}}
\def\dzbar{{\overline{dz}}}
 \def\ii{\mathrm{i}}
  \def\d{\partial}
 \def\dbar{{\overline{\partial}}}
\def\dzbar{{\overline{dz}}}
\def \ddbar {\partial \overline{\partial}}
\def\cN{\mathcal{N}}
\def\cE{\mathcal{E}}  \def\cO{\mathcal{O}}
\def\cF{\mathcal{F}}
\def\cS{\mathcal{S}}
\def\cQ{\mathcal{Q}}
\def\P{\mathbb{P}}
\def\cI{\mathcal{I}}
\def \loc{\mathrm{loc}}
\def \cC{\mathcal{C}}
\bibliographystyle{plain}
\def \dim{\mathrm{dim}}
\def \Sing{\mathrm{Sing}}
\def \Id{\mathrm{Id}}
\def \rank{\mathrm{rank}}
\def \tr{\mathrm{tr}}
\def \Ric{\mathrm{Ric}}
\def \Vol{\mathrm{Vol}}
\def \RHS{\mathrm{RHS}}
\def \liminf{\mathrm{liminf}}
\def \ker{\mathrm{Ker}}
\def \nn{\mathrm{nn}}
\def \Cut{\mathrm{Cut}}
\def \Pic{\mathrm{Pic}}
\def \Alb{\mathrm{Alb}}
\def \diam{\mathrm{diam}}
\def \Aut{\mathrm{Aut}}
\def \Sym{\mathrm{Sym}}
\def \reg{\mathrm{reg}}
\def\stateparagraph{\vskip7pt plus 2pt minus 1pt\noindent}
\begin{abstract}
In the first part of this note, we discuss the compact K\"ahler manifold with a strongly pseudo-effective tangent bundle.
In the second part, we give new proof of the fact that the only projective manifolds with the big tangent bundle are the projective spaces.
In the third part, we give a characterisation of nef vector bundle.
\end{abstract}
\section{Introduction}
The compact Riemann surfaces that are not of general type are those of semi-positive curvature. In higher dimensions, the situation is much more subtle, and it has been studied in various possible generalisations (e.g., when the manifold is projective or compact K\"ahler, or with the positivity assumption on the whole tangent bundle or only on the anticanonical line bundle).
This kind of problem has been studied by many authors and in many works, cf. \cite{Mori79},\cite{SY80},  \cite{Mok88}, \cite{CP91}, \cite{DPS93}, \cite{DPS94},  \cite{CH17}, \cite{CH19},\cite{Cao19}, \cite{CCM}   among others.

This note is a continuation of the work of \cite{Wu20}.
The general setting of this paper is the following.
Let $(X, \omega)$ be a compact K\"ahler manifold, and we study the
manifolds with strongly psef tangent bundle and properties of strongly psef vector bundle.
The definition of strongly psef vector bundle over $X$ is introduced in \cite{BDPP}.
\stateparagraph
{\bf Definition A.} {\it
Let $(X, \omega)$ be a compact K\"ahler manifold and $E$ a holomorphic vector bundle on $X$. 
Then $E$ is said to be strongly pseudo-effective (strongly psef for short) if the line bundle $\cO_{\P(E)}(1)$ is pseudo-effective on the projectivized bundle $\P(E)$ of hyperplanes of $E$, i.e.\ if for every $\varepsilon>0$ there exists a singular metric $h_\varepsilon$ with analytic singularities on $\cO_{\P(E)}(1)$ and a curvature current $i \Theta(h_\varepsilon) \geq - \varepsilon \pi^* \omega$, and if the projection $\pi(\Sing(h_\varepsilon))$ of the singular set of~$h_\varepsilon$ is not equal to~$X$.
}

In the definition of the strongly pseudo-effective vector bundle, an additional condition is made on the approximate singular metrics on the tautological line bundle.
One may ask whether the additional condition could be made directly on some positive singular metric on the tautological line bundle.
For example, let $h$ be a positive singular metric on $\cO_{\P(E)}(1)$, could we consider the condition that the projection of singular set $\{h = \infty\}$ in $X$ is a (complete) pluripolar set?
However, this kind of condition does not behave well functorially as shown in Example 5 suggested to the author by Demailly.

This note separates into three independent parts.

The first part concerns the fundamental group of a compact K\"ahler manifold with a strongly psef tangent bundle.
In fact, we show the following result.

\stateparagraph
{\bf Theorem A.} {\it 
Assume that the fundamental group of any compact K\"ahler manifold with a strongly psef tangent bundle has subexponential growth.
Up to a finite \'etale cover, the Albanese map of a compact K\"ahler manifold with a strongly psef tangent bundle is a submersion with rational connected very general fibre.
Moreover, the fundamental group of a compact K\"ahler manifold with a strongly psef tangent bundle is almost abelian.
In particular, the structural result on the Albanese map and the fundamental group holds for a compact K\"ahler manifold $X$ such that $T_X$ is strongly psef and $-K_X$ is nef.
}

In other words, if one can show that the fundamental group of compact K\"ahler manifold with strongly psef tangent bundle is always not so ``big,"
then we have the structure theorem, which implies easily that the fundamental group is in fact almost abelian.
In fact, since a rational connected K\"ahler manifold has a trivial fundamental group, the homotopy exact sequence shows that the fundamental group of this finite cover is abelian.
This part can be seen as the generalisation of the Main theorem of \cite{DPS94} from the nef case to the strongly psef case under the assumption of the fundamental group.
This kind of problem is raised for example in \cite{Mat18}, \cite{Mat21}, \cite{MW21}.

Notice that for the special case when the anticanonical bundle is in addition nef, the result is known by the work of P\u{a}un \cite{Pau98a} using profound results of Cheeger-Colding \cite{CC96}, \cite{CC97}, \cite{KW11}.

When the manifold is projective, it is a direct consequence of the work of \cite{HIM}.
An essential tool in their proof is the results of \cite{BDPP} and \cite{GHS}, which state that the canonical bundle of the MRC quotient is pseudo-effective.
(The projectivity assumption is also needed to apply the result of \cite{Hor07}.)
It is conjectured that the results of \cite{BDPP} are still valid in the compact K\"ahler setting.
However, it seems to be far of reach at the moment.

Our basic observation to treat the compact K\"ahler setting is that assuming the geometric situation remaining unchanged from the projective case to the compact K\"ahler case, the very general fibre of the Albanese morphism of some finite \'etale cover is a rational connected compact K\"ahler manifold.
However, by the fundamental work of \cite{Cam}, such a very general fibre is necessarily projective.
Thus we can apply the result of \cite{HIM} to this very general fibre.
It turns out that the rational connectedness of this very general fibre is equivalent to the projectivity, which is much easier to prove assuming that the fundamental group is not ``too big".

More precisely, we are reduced to the following situation.
Let $X$ be a compact K\"ahler manifold with a strongly psef tangent bundle.
Let $\tilde{X}$ be a finite \'etale cover of $X$ of maximum irregularity $q=q(\tilde{X}) =h^1(\tilde{X},\cO_{\tilde{X}})$.
Assume $q=0$.
For simplicity, assume that $X$ attains the maximum irregularity.
Then we need to prove that $X$ is projective.
It is shown by the non-existence of non-trivial holomorphic two-form on $X$. By absurd,
a non-trivial holomorphic two-form will induce a numerically flat vector bundle over $X$.
Now we can follow the strategy of \cite{DPS94} to deduce a contradiction.
By the work of \cite{DPS94} (also cf. \cite{Deng}), a numerical trivial vector bundle is an extension of Hermitian flat vector bundles.
In particular, we have a representation of $\pi_1(X)$ into some unitary group.
By Tits alternating theorem, the image of the representation either contains a non-abelian free group or is an almost nilpotent group.
We can show that the first Betti number of some finite \'etale cover is strictly positive for the second case.
It contradicts the fact that $q=0$.
It remains to show that the first case can not happen.
In the case of nef case, which is proven in \cite{DPS93},
the authors prove that the fundamental group of a compact K\"ahler manifold with nef anticanonical line bundle has subexponential growth.
Thus it cannot have a non-abelian free subgroup.
For the moment, we do not know how to generalise their volume estimate in our case since we should not hope for the existence of a smooth metric with an arbitrarily small Ricci curvature lower bound. 
Thus we state the fundamental group estimate as a conjecture and as an assumption in Theorem A.
\stateparagraph
{\bf Conjecture A.} {\it 
The fundamental group of a compact K\"ahler manifold with a strongly psef tangent bundle has subexponential growth.
}

Recall the definition of subexpotenial growth of a group (cf. \cite{DPS93}).
If $G$ is a finitely generated group with generators $g_1,\cdots,g_p$, we denote by $N(k)$ the number of elements $\gamma \in G$ which can be written as words
$$\gamma=g_{i_1}^{\varepsilon_1} \cdots g^{\varepsilon_k}_{i_k},  \varepsilon_j \in \{ 0,1,-1\}$$
of length $\leq k$ in terms of the generators.  The group $G$ is said to have subexponential growth if for every $\varepsilon >0$ there is a constant $C(\varepsilon)$ such that $N(k) \leq C(\varepsilon)e^{\varepsilon k}$ for $k\geq 0$.

This notion is independent of the choice of generators.
It follows easily that such a fundamental group does not have a non-abelian free group.
Apply this conjecture to avoid the non-abelian free group in the image of $\pi_1(X)$ (in the representation).

Notice that the conjecture holds if the manifold is projective by the work of \cite{HIM}.
Notice that this conjecture is also optimal because there is a projective manifold with a psef anticanonical line bundle whose fundamental group contains a non-abelian free group.
An example is discussed in Section 2.

Inspired by the case when the tangent bundle of a compact K\"ahler manifold is nef, we have the following conjecture.
\stateparagraph
{\bf Conjecture B.} {\it
Let $X$ be a compact K\"ahler manifold with a strongly psef tangent bundle.
Let $\tilde{X}$ be a finite étale cover of $X$ such that $\tilde{q}(X)=q(\tilde{X})$.
Then the Albanese map is a local isotriviality onto the Albanese torus with rationally connected fibres.

}
Compared to the nef case, the main difficulty is that the general fibre is not necessarily Fano. 
In particular, a priori, a high power of the anticanonical line bundle does not necessarily embed the fibres in a fixed projective space which is a standard strategy to prove the isotriviality.

Based on the work of Brion and Fujiki, we verify the conjecture for compact quasi-homogeneous manifolds of class $\cC$ (i.e. bimeromorphic to compact K\"ahler manifolds).
\stateparagraph
{\bf Proposition A.}(Remark 5) {\it
Let $X$ be a compact quasi-homogeneous manifold of class $\cC$ (i.e. having an open and dense orbit under the natural $\Aut_0(X)-$action).
Then the Albanese map of $X$ is locally isotrivial whose fibre $F$ has a big anti-canonical divisor.
Moreover, the universal cover of $X$ is isomorphic to $\C^{\tilde{q}(X)} \times F$.
}

The example of the Hopf surface shows that the proposition does not hold for arbitrary compact homogeneous manifolds.

However, when the tangent bundle is strongly big (cf. the definition below), the Albanese torus reduces to a point from which we infer that the manifold is rationally connected.
In fact, the only possibility is the projective space as shown by \cite{Iwe}, \cite{FM}.
In Section 3, we give new proof of this fact using Serge currents.
\stateparagraph
{\bf Theorem B.} {\it
Let $X$ be a compact K\"ahler manifold of dimension $n$ with a strongly big tangent bundle.
Then $X$ is biholomorphic to $\P^n$.
}

Here the notion of strong bigness is defined as follows.
\stateparagraph
{\bf Definition B.} {\it
Let $E$ be a vector bundle over a compact K\"ahler manifold $X$. 
$E$ is said to be strongly big if there exists a K\"ahler current $T$ in the class of $c_1(\cO_{\P(E)}(1))$ such that the projection of the singular part $E_{+}(T)=\{\nu(T,x) >0\}$ does not cover $X$.
Here $\nu(T,x)$ is the Lelong number of $T$ at $x \in \P(E)$.
}

According to the knowledge of the author, this is first defined in \cite{Iwe} in a different form.
In Section 3, we will prove that the two definitions are equivalent.


The third part of this work is a generalisation of the characterisation of the nef line bundle of P$\check{a}$un to the higher rank case (cf. \cite{Pau98b}).
\stateparagraph
{\bf Theorem C.} {\it
Let $E$ be a vector bundle over a compact K\"ahler manifold $(X, \omega)$ of dimension $n$.
$E$ is nef if and only if for any irreducible analytic subset $Z \subset X$, $E|_Z$ is a strongly psef vector bundle.
}

In particular, a vector bundle over a compact K\"ahler space which does not contain any proper analytic subset other than points (for example, a very general torus) is strongly psef if and only if it is nef.

Here, a vector bundle over an irreducible compact K\"ahler complex space $Z$ is called strongly psef if a modification of $Z$ is smooth such that the pullback of the vector bundle is strongly psef.
Similarly, a torsion-free sheaf is called psef if there exists a modification of $Z$ which is smooth such that the pullback of the coherent sheaf modulo torsion is a strongly psef vector bundle.

This result is optimal in the sense that we can not change the strong psefness condition in the equivalent condition by the weak psefness condition.

At the end of the note, we generalise the Main Theorem of \cite{Wu20} to the compact manifolds of the class $\cC$.
\stateparagraph
{\bf Proposition B.}{\it
Let $X$ be a compact manifold of the Fujiki class $\cC$.
Let $E$ be a holomorphic vector bundle over $X$.
Then the following conditions are equivalent:
\begin{enumerate}
\item $E$ is nef and $c_1(E)=0$;
\item $E$ is strongly psef and $c_1(E)=0$;
\item $E$ admits a filtration by vector bundles whose graded pieces are hermitian flat, i.e. given by unitary representations of $\pi_1(X)$.
\end{enumerate}
}
As a geometric consequence, we have the following corollary.
\stateparagraph
{\bf Corollary A.}{\it
Let $X$ be a compact connected manifold of the Fujiki class $\cC$ such that $c_1(X)=0$.
Assume that there exists $q \neq n:=\mathrm{dim}_\C X$ such that $\wedge^q T_X$ is strongly psef.
Then up to a finite \'etale cover, $X$ is a complex torus.
In particular, an irreducible symplectic, or Calabi-Yau manifold of the Fujiki class $\cC$ does
not have strongly psef $\wedge^q T_X$ or $\wedge^q T^*_X$ for any $q \neq n:=\mathrm{dim}_\C X$.
}

In the Appendix, based on an example of Ueno \cite{Uen}, we discuss the surjectivity of Albanese morphism of a compact manifold with nef anticanonical line bundle.
\paragraph{}
\textbf{Acknowledgement} I thank Jean-Pierre Demailly, my PhD supervisor, for his guidance, patience, and generosity. 
I would like to thank Junyan Cao, Patrick Graf, Mihai P\u{a}un, and Thomas Peternell for some very useful suggestions on the previous draft of this work.
In particular, I warmly thank Michel Brion for helping to classify the quasi-homogeneous case.
I would also like to express my gratitude to my colleagues of Institut Fourier for all the interesting discussions we had. This work is supported by the European Research Council grant ALKAGE number 670846, managed by J.-P. Demailly.

\section{Fundamental group of manifolds with strongly psef tangent bundle}
In this section, we prove that the fundamental group of a compact K\"ahler manifold with a strongly psef tangent bundle is almost abelian under some assumptions.
The basic fact is the variant of projective case proven in \cite{HIM} based on \cite{BDPP} and \cite{GHS}.
\begin{myprop}
Let $X$ be a projective manifold such that $T_X$ is strongly psef.
Let $\tilde{X}$ be a finite \'etale cover of maximum irregularity $ q(\tilde{X} ) = \tilde{q}(X)$.
Then the Albanese morphism of $\tilde{X}$ is a smooth fibration such that the very general fibre $F$ has a strongly psef tangent bundle and is rational connected.
Moreover $\pi_1(\tilde{X}) \cong \Z^{2\tilde{q}(X)}$.
\end{myprop}
\begin{proof}
The only new statement compared to Theorem 3.11 of \cite{HIM} is that under the projective condition, the very general fibre is rational connected and the fundamental group calculation.
The proof follows from the proof of Theorem 3.14 and 3.15 in \cite{DPS94}.
Since we will also follow this strategy in the compact K\"ahler setting, we only point out the difference here and give the full details in the compact K\"ahler setting.
To avoid the additional assumption on the fundamental group,
instead of using alternative 3.10 of \cite{DPS94}, we use the following lemma for the very general fibre with the strongly psef tangent bundle.
Remind that a rationally connected manifold is also simply connected.
The arguments of the homotopy groups thus apply.
\end{proof}
\begin{mylem}
Let $X$ be a projective manifold such that $T_X$ is strongly psef.
Assume $\tilde{q}(X)=0$.
Then $X$ is rationally connected. 
\end{mylem}
\begin{proof}
By theorem 3.10 of \cite{HIM} (based on the work of \cite{Hor07}), there exists a surjective morphism $f: X \to Y$ such that $Y$ admits a finite \'etale cover $A \to Y$ by an abelian variety $A$ and that the general fibre of $f$ is rational connected.

We claim that $Y$ is a point.
If it is not the case, $A$ is an abelian variety of positive dimensions.
There exists a base change
$$
\begin{tikzcd}
\tilde{X} \arrow[r] \arrow[d,"\pi"]
& X \arrow[d] \\
A \arrow[r]
& Y
\end{tikzcd}
$$
such that $\tilde{X}$ is a projective manifold with $T_{\tilde{X}}$ strongly psef.
Then $H^0(\tilde{X}, \Omega^1_{\tilde{X}}) \neq 0$ since it contains
$\pi^* H^0(A, \Omega^1_{A})$
and $f$ (hence $\pi$) is surjective.
This contradicts the condition that $\tilde{q}(X)=0$.
\end{proof}
In the following, we generalize the above lemma to the compact K\"ahler case with additional assumptions.
First, we have the following observation.
\begin{mylem}
Let $X$ be a simply connected compact K\"ahler manifold such that $T_X$ is strongly psef.
Then $X$ is rationally connected if and only if $X$ is projective. 
\end{mylem}
\begin{proof}
The necessary part is just the above lemma.
The sufficient part is a direct consequence of the following important result of Campana (Corollary on page 212 \cite{Cam});

{\it 
A rational connected compact K\"ahler manifold is in fact projective.
}
\end{proof}
Now we shall show that the manifolds in the above lemma are always projective.
Then we will consider the manifolds $X$ with a strongly psef tangent bundle and $\tilde{q}(X)=0$.
Before considering the general case, we consider the case of dimension 2. 
By assumption, $-K_X$ is psef by Corollary 1 in \cite{Wu20}.
$c_1(-K_X) \neq 0$ otherwise it is a complex torus by Corollary 9 of \cite{Wu20}.
In fact, a compact K\"ahler manifold with a strongly psef tangent bundle and trivial first Chern class admits a finite \'etale cover of a complex torus.
Combining these two facts for the non-torsion psef line bundle $-K_X$, for any $m >1$, $H^0(X, mK_X)=0$.
Since $H^{0,1}(X)=0$, the surface is rational by the classical Castelnuovo theorem.

We will need the following lemma analogue to Lemma 1.16 in \cite{DPS94}.
\begin{mylem}
Let $E$ be a strongly psef vector bundle over a connected compact K\"ahler n-fold $(X, \omega)$ and let $\sigma \in H^0(X,E^*)$ be a non-zero section. Then $\sigma$ does not vanish anywhere.
\end{mylem}
\begin{proof}
Since $E$ is strongly psef, there exists a sequence of quasi-psh functions $w_m(x, \xi)=\log( |\xi|_{h_m})$ with analytic singularities induced from Hermitian metrics $h_m$ on $\Sym^m E^*$ 
such that
the singularity locus projects onto a proper Zariski closed set $ Z_m$ in $X$,
and 
$$i \d \dbar w_m \ge -m\varepsilon_m p^*\omega $$ in the sense of currents
with $\lim \varepsilon_m = 0$.
Here $p: \Sym^m E^* \to X$ is the projection.

The non-zero section induces a holomorphic map $X \to S^m E^*$.
The pullback of $i \d \dbar w_m$ is well defined since everything is smooth on a non-empty Zariski open set.
Divide these currents by $m$ respectively.
In conclusion, there exists a sequence of currents $T_m$ on $X$ such that for any $\varepsilon >0$, $T_m \geq -\varepsilon \omega$ for $m \geq m_0(\varepsilon)$ large enough.
Notice that by construction, $T_m$ are exact.
Thus we have
$$\int_X (T_m+\varepsilon \omega) \wedge \omega^{n-1}=\varepsilon \int_X \omega^n.$$
In particular, up to taking a subsequence $T_m$ tends to 0 in the sense of currents as $m$ tends to infinity.

If $\sigma(x) = 0$ at some point $x \in X$, the Lelong number $\nu (T_m,x)$ would be at least equal to the vanishing order of $\sigma$ at $x$. 
By the semi-continuity of the Lelong number \cite{Siu74},
the Lelong number of any weak limit of $T_m$ would have a strictly positive Lelong number at $x$.
A contradiction.
\end{proof}
This result is entirely false if we only assume that the vector bundle is psef in the weak sense.
An easy counter-example is taking $E$ to be $\cO(A) \oplus \cO(-A)$ with a very ample divisor $A$.
The zero set of any global section of $E^*$ is never empty.

Now we can prove the simply connected case.
\begin{mythm}
Let $X$ be a simply connected compact K\"ahler manifold such that $T_X$ is strongly psef.
Then $X$ is rationally connected. 
\end{mythm}
\begin{proof}
By Lemma 2, we reduce to prove the projectivity of $X$ by showing that $H^{0,2}(X)=0$.

Assume that it is not the case.
Let $\omega \in H^0(X,  \Omega^2_X)$ be a non zero section.
Let $q$ be the maximal natural number such that $\omega^q \neq 0$ but $\omega^{q+1} =0$.
Since $\wedge^q T_X$ is strongly psef,
by the above lemma 3,
$\omega^q$ does not vanish anywhere.
Define a foliation on $X$ in the following way
$$\cF:=\{\xi \in T_{X,x}, i_{\xi} \omega(x)=0\}$$
where $i_\xi$ is the contraction.
In local coordinates, $\cF$ is given by the solution of the linear system $A(z) \xi=0$ where $A(z)$ is a matrix of local holomorphic functions.
Since $\omega^q$ does not vanish anywhere,  at any point, there exists an invertible minor of rank $q$ at this point.
$\omega^{q+1}=0$ implies that any minor of rank $q+1$ is not invertible at any point.
In particular, the rank of $\cF$ is constant over $X$, implying that $\cF$ is a regular foliation.

By the very definition of $\cF$, $\omega$ induces a non-degenerated 2 form over the quotient bundle $T_X/\cF$, which varies holomorphically.
In particular, $T_X / \cF$ is holomorphically isomorphic to its dual bundle.
Thus we have that $c_1(T_X / \cF)=0$.

Since $T_X$ is strongly psef, $T_X/ \cF$ is also strongly psef.
By the main theorem of \cite{Wu20}, $T_X / \cF$ is a numerical flat vector bundle.
By Theorem 1.18 of \cite{DPS94},
$T_X / \cF$ admits a filtration
$$\{0\}=E_0 \subset E_1 \subset \cdots \subset E_p=T_X / \cF$$by vector subbundles such that the quotients $E_k/E_{k-1}$ are hermitian flat, i.e. given by unitary representations $\pi_1(X) \to U(r_k)$.
Since $X$ is simply connected, the representations are trivial, which means that $E_k / E_{k-1}$ are trivial vector bundles.

In particular, there exists a non-zero section of the dual bundle $T_X / \cF$.
On the other hand, $(T_X / \cF)^*$ is a subbundle of $\Omega^1_X$.
Thus $H^0(X, \Omega^1_X) \neq 0$
which contradicts the assumption that the manifold is simply connected.
\end{proof}
\begin{mylem}
Let $X$ be a compact K\"ahler manifold such that $T_X$ is strongly psef with a finite fundamental group.
Then $X$ is rationally connected.
In particular, $X$ is simply connected.
\end{mylem}
\begin{proof}
By assumption, there exists a finite \'etale cover $\tilde{X}$ of $X$ such that $T_{\tilde{X}}$ is strongly psef and $\tilde{X}$ is simply connected.
By Theorem 1, $\tilde{X}$ is rationally connected.
In particular, $\chi(\tilde{X}, \cO_{\tilde{X}})=1$
since $H^0(\tilde{X}, \Omega^p_{\tilde{X}})=0(\forall p \geq 1)$
(cf. for example 4.18 of \cite{Deb}).
Since $\chi(\tilde{X}, \cO_{\tilde{X}})$ is equal to the degree of cover times $\chi(X, \cO_X)$, the degree is 1.
In other words, $\tilde{X} = X.$
\end{proof}
Now we generalise the rational connectedness to the compact K\"ahler manifolds $X$ with strongly psef tangent bundle and $\tilde{q}(X)=0$ assuming Conjecture A.


\begin{myprop}
Let $X$ be a compact K\"ahler manifold such that $T_X$ is strongly psef with $\tilde{q}(X)=0$.
Assume that the fundamental group of $X$ has subexponential growth.
Then $X$ is rationally connected.
\end{myprop}
\begin{proof}
The proof follows the outline of the proof of Theorem 1 initially from Proposition 3.10 of \cite{DPS94}.
As proof of Theorem 1, it is sufficient to show the projectivity by proving that $H^0(X, \Omega^2_X)=0$.

The only place in the proof of Theorem 1, which uses the assumption that $X$ is simply connected, is by showing that the quotient bundles $E_1$ are trivial.
The hermitian flat vector bundle $E_1$ is equivalent to a representation of $\pi_1(X) \to U(r)$ where $r$ is the rank of $E_1$.

By the alternative theorem of Tits \cite{Ti72}, the image of $\pi_1(X)$ either contains a non-abelian free subgroup or a solvable subgroup of finite index.
By our assumption, $\pi_1(X)$ has subexponential growth and cannot contain a non-abelian free subgroup.  

In conclusion, the image of $\pi_1(X)$  contains a solvable subgroup denoted by $\Gamma$ of finite index.
Let 
$$0=\Gamma_N \subset \Gamma_{N-1} \subset \cdots \subset \Gamma= \Gamma_0 $$
be the series of derived subgroups $\Gamma_{i+1}= [\Gamma_i,\Gamma_i]$.

If the image of $\pi_1(X)$ is finite, the pull-back of $E_1$ by some finite \'etale cover $p: \tilde{X} \to X$ gives rise to a trivial subbundle $p^* E_1 \subset p^* \Omega^1_X= \Omega^1_{\tilde{X}}$ as in the proof of Theorem 1.
In particular, $H^0(\tilde{X}, \Omega^1_{\tilde{X}}) \neq 0$ which induces a contradiction.

Otherwise, the image of $\pi_1(X)$ and $\Gamma$ is infinite.
Let $i$ be the smallest index such that $\Gamma_i / \Gamma_{i+1}$ is infinite. 
Then $G / \Gamma, \Gamma/ \Gamma_1, \cdots, \Gamma_{i-1}/ \Gamma_i$ are finite, and so is $G / \Gamma_i$.
The pre-image of $\Gamma_i$ (denoted by $\Pi_i$) is of finite index in  $\pi_1(X)$ and gives rise to a finite  \'etale cover
$p: \tilde{X} \to X$ with $\pi_1(\tilde{X}) \cong \Pi_i$.
The representation of $\pi_1(X)$ would induce a surjective homomorphism
$$H_1(\tilde{X}, \Z)= \Pi_i / [\Pi_i, \Pi_i] \to \Gamma_i / [\Gamma_i, \Gamma_i]=\Gamma_{i+1}.$$
In particular, $H_1(\tilde{X}, \Z)$ is infinite which implies that the first Betti number of $\tilde{X}$ is not zero.
By Hodge decomposition theorem, it implies that $H^0(\tilde{X}, \Omega^1_{\tilde{X}}) \neq 0$
which contradicts the fact that $\tilde{q}(X) =0$.
\end{proof}
Now we can extend Proposition 1 to the compact K\"ahler case.
\begin{mythm}
Let $X$ be a compact K\"ahler manifold such that $T_X$ is strongly psef.
Assume that Conjecture A holds.
Let $\tilde{X}$ be a finite \'etale cover of maximum irregularity $q= q(\tilde{X} ) = \tilde{q}(X)$.
Then the Albanese morphism of $\tilde{X}$ is a smooth fibration such that the very general fibre $F$ has a strongly psef tangent bundle and is rational connected.
Moreover $\pi_1(\tilde{X}) \cong \Z^{2\tilde{q}(X)}$.
More precisely, as an abstract group, the fundamental group $\pi_1 (X)$ is an extension of a finite group by $\Z^{2q}$.
In particular, the structural result holds for a compact K\"ahler manifold $X$ such that $T_X$ is strongly psef and $-K_X$ is nef.
\end{mythm}
\begin{proof}
It is done by induction on the dimension of $X$.
The proof is essentially parallel to Theorem 3.14 and 3.15 of \cite{DPS94}.
For the convenience of the reader, we give the proof here.

Suppose that the result is known in dimension $< n$.
The Albanese morphism is a submersion since if some non-trivial linear combination of the basis of $H^0(\tilde{X}, \Omega^1_{\title{X}})$ vanishes at some point, by Lemma 3, it will vanish identically.
It contradicts the fact that it is a basis.
The relative tangent bundle sequence
$$0 \to T_{\tilde{X}/\Alb(\tilde{X})} \to T_{\tilde{X}} \to \alpha^* T_{\Alb(\tilde{X})} \to 0$$
with Albanese morphism $\alpha: \tilde{X} \to \Alb(\tilde{X})$
in which $\alpha^* T_{\Alb(\tilde{X})}$ is trivial shows by Corollary 4 of \cite{Wu20} that $T_{\tilde{X}/\Alb(\tilde{X})}$ is strongly psef.

The very general fibre of $\alpha$ not contained in the projection of the non-nef locus of $\cO(1)$ over $\P(T_{\tilde{X}/\Alb(\tilde{X})})$ 
has also strongly psef tangent bundle.
By the universal property of Albanese morphism, the fibres are also connected (cf. Proposition 3.9 of \cite{DPS94}).
If the dimension of a very general fibre $F$ is $n$,
the manifold $X$ satisfies that $\tilde{q}(X)=0$.
By Proposition 2, $X$ is a rational connected projective manifold.
Then by Proposition 1 (the projective case), we conclude the fundamental group statement. 

Otherwise, by the induction hypothesis,  the fundamental group of a very general fibre $F$ contains an abelian subgroup of finite index.
By proposition 3.12 of \cite{DPS94},
$\tilde{q}(F)=0$. Thus $F$ is a rational connected projective manifold.
The homotopy exact sequence of the Albanese fibration shows that
$$\pi_1(\tilde{X}) \cong \pi_1 (A( \tilde{X}))
\cong \Z^{2q} .$$
 There is a composition of finite étale covers $\tilde{\tilde{X}} \to \tilde{X} \to X$ such that $\pi_1(\tilde{\tilde{X}})$ is a normal subgroup of finite index of $\pi_1(X)$ by taking the intersection of
all conjugates of $\pi_1(\tilde{X})$ in $\pi_1(X)$.
Again $ \pi_1(\tilde{\tilde{X}})
\cong \Z^{2q}$ as a subgroup of finite index in
$\pi_1 (X) \cong \Z^{2q}$.
Since the composition of the finite \'etale covers is Galois, it induces an extension of $\pi_1(X)$ by a finite group and $\Z^{2q}$.

The last statement follows easily from the fact that if $-K_X$ is nef, the fundamental group has subexponential growth and that the very general fibre of the Albanese morphism still has a nef anticanonical line bundle.
Notice that in this case, we can prove that the fundamental group is almost abelian, as shown by $\cite{Pau98a}$.
However, the proof is much more involved using the profound results of Cheeger-Colding.
\end{proof}
Notice that the same proof works assuming that the fundamental groups of $X$ and $F$ have subexponential growth instead of assuming that Conjecture A holds for any compact K\"ahler manifold with a strongly psef tangent bundle.
 
It seems that the Albanese morphism is in fact isotrivial up to finite \'etale cover.
However, unlike the nef tangent bundle case, the fibres of the Albanese morphism are not necessarily Fano.
Thus $-K_X$ is not relatively ample to embed the fibres in a numerically flat way in some fixed projective space.
In \cite{HIM}, it is asked whether $-K_X$ is relatively big or not.

In the case of (non-necessarily projective, but K\"ahler) surface, the isotriviality of the Albanese morphism up to a finite \'etale cover can be proven.
Up to a finite \'etale cover, 
we can assume that $X$ attains the maximum irregularity.
We separate the discussion into three cases regarding the dimension of the Albanese torus of the surface.

If the Albanese torus is of dimension 2, then the fibres are of 0 dimensional, which are connected.
Thus up to a finite \'etale cover, the surface is a torus.

If the Albanese torus is of dimension 0, 
as we have seen, the surface is rationally connected.
It is equivalent to saying that the surface is rational for a surface.
This is studied in \cite{HIM}. 
For example, they showed that the blow-up of the Hirzebruch surface along the general three points has a strongly pseudo-effective tangent bundle.
However, it is unclear to give a complete classification.

If the Albanese torus is of dimension 1, then we claim that the surface is a ruled surface over the elliptic curve.
First, we show that the surface is minimal, i.e. it can not be the blow-up of some smooth surface $Y$ at some point $p$.
Assuming this is the case, then we have the following commutative diagram.
$$
\begin{tikzcd}
X \arrow[r,"\pi"] \arrow[d,"\alpha_X"]
& Y \arrow[d,"\alpha_Y"] \\
\Alb(X) \arrow[r,"\varphi"]
& \Alb(Y)
\end{tikzcd}
$$
Since $\pi$ is a blow-up, we have linear isomorphism $\pi^*: H^0(Y, \Omega^1_Y) \to H^0(X, \Omega^1_X)$
which implies that $\varphi: \Alb(X) \to \Alb(Y)$ is an isomorphism.
Since $\alpha_X$ is a submersion, $\alpha_Y$ is also a submersion.
In particular, all the fibres of $\alpha_Y$ are diffeomorphic.
Denote by $F$ the fibre of $\alpha_Y$.
The general fibre of $\alpha_X$ is diffeomorphic to $F$.
The exceptional divisor $E$ is contained in a fibre of $\alpha_X$.
In fact,  the restriction of $\alpha_X$ on the exceptional divisor can be lifted to a holomorphic map to $\C$ and thus, it is a constant map.
This gives a special fibre of $\alpha_X$ topologically as the wedge sum of $E$ and $F$.
The reduced homology group of $H_*(E \vee F)$ is isomorphic to $H_*(F) \oplus H_*(S^2)$.
It contradicts that the fibres of $\alpha_X$ have to be diffeomorphic.

If the surface is not a ruled surface, the line bundle $K_X$ is nef.
$T_X$ is strongly psef implies that $-K_X$ is psef.
In particular, $c_1(X)=0$.
By the Beauville-Bogomolov decomposition theorem, up to finite cover, the surface is either a complex torus of dimension two or a direct sum of the torus of dimension 1.
Both cases contradict the dimension of the Albanese torus.

By Theorem 1.5 of \cite{HIM}, for a ruled surface over the elliptic curve, the tangent bundle is strongly psef.
In fact, we can give a complete list of classifications in this case.
In particular, the Albanese morphism is locally isotrivial up to a finite \'etale cover.

In fact, we can hope for some stronger structural results to be hold. 
Let $X$ be a compact Kähler manifold with a strongly psef tangent bundle.
Then the universal cover of $X$ is hoped to be a product of $\C^q$ and $F$
where $q=\tilde{q}(X)$ and $F$ is a rational connected compact manifold with a strongly psef tangent bundle with a big anticanonical line bundle.
\begin{myrem}
{\em 
Notice that the K\"ahler condition is not necessarily in the definition of strongly psef vector bundle and strongly psef torsion-free sheaf.
The usual algebraic properties of strongly psef vector bundles (resp. torsion-free sheaves) hold over arbitrary compact complex manifolds.
Of course, one main technical tool (Serge currents) is unclear to be held over arbitrary compact complex manifolds by lack of Stokes formula (which is the reason we always suppose the manifold to be Kähler in \cite{Wu20}).
}
\end{myrem}
Thus one can ask whether we have the same list of compact complex manifolds with strongly psef tangent bundles without assuming the Kähler condition.
In general, this is false, as shown in the case of a quasi-homogeneous manifold.
Let us recall some basic definitions.
\begin{mydef}(Quasi-homogeneous manifold)

We fix a complex connected Lie group $G$  acting holomorphically on a compact complex manifold $X$.
$X$ is called quasi-homogeneous if $X$ contains an (in fact unique) open dense orbit.
\end{mydef}
\begin{myex}
{\em 
In this example, we show that a quasi-homogeneous manifold has a strongly psef tangent bundle, and then we give examples of smooth Moishezon quasi-homogeneous manifolds which are not projective.
In particular, these give examples of a non-Kähler manifold with a strongly psef tangent bundle.

In fact, the fact that the tangent bundle is strongly psef is a direct consequence of the well-known fact that the tangent bundle of a quasi-homogeneous manifold is generically globally generated. 
Denote by $G\cdot x_0$ the Zariski dense open orbit of $X$.
The group action induces a morphism of analytic group schemes $G \to \Aut(X)$ whose differential at identity is a morphism $\mathfrak{g} \to H^0(X, T_X)$.
More precisely, the vector field is defined as follows: for any $\xi \in \mathfrak{g}$, 
$$v_\xi (x)= \frac{d}{dt}|_{t=0} \exp(t \xi)\cdot x.$$
Since $G/ G_{x_0} \cong G \cdot x_0$ and $G \cdot x_0$ is an open neighbourhood of $x_0$, $v_\xi(x_0)$ generates any element of $T_{x_0,X }$ for appropriate elements of $\xi \in \mathfrak{g}$.
Since the action of $G$ on $G\cdot x_0$ is transitive, $T_X$ is hence globally generated on $G\cdot x_0$.

Thus the global sections induce a psh function on the total space of $S^m T_X$ for any $m >0$ with analytic singularities.
The natural projection of the singular part is contained in $X \setminus G\cdot x_0$ which concludes that $T_X$ is strongly psef.

Now take $G$ to be $(\C^*)^{\oplus r}$ (i.e. $X$ is a toric variety).
By the work of \cite{Mor96}, we know that for any toric variety, there exists a bimeromorphic model which is projective up to a finite composition of blow-ups of smooth toric centres. 
In particular, any toric variety is Moishezon.

On the other hand, by the work of \cite{Dan78}, a toric variety is projective if and only if its associated fan is strongly polytopal.
It is well known that non strongly polytopal cone exists in dimensions at least 3.
In particular, the corresponding toric variety is not projective.
We can even choose the fan such that the corresponding toric variety is non-singular. This only requires that any face $\sigma$ is non-singular in the following sense:
there exist a $\Z-$basis of the fan $\{n_1, \cdots, n_r\}$ and $s \leq r$ such that $\sigma=\R_{\geq 0} n_1 + \cdots + \R_{\geq 0} n_s$. 
 It gives an example of a smooth Moishezon toric variety $X$, which is not projective.

Notice by theorem 4.1 of \cite{DPS94}, the tangent bundle of such $X$ is not nef since a Moishezon manifold with nef tangent bundle is necessarily projective.

}
\end{myex}
Notice that such manifolds are of dimension at least 3.
Let us consider the case of compact surface $X$ with a strongly psef tangent bundle.
As shown in the example of the blow-up of $\P^2$, it seems to be very difficult in general.
Let us consider the case of minimal surface.
\begin{myex}(minimal surface)
{\em 

First, notice that if $\kappa(X) \geq 0$, the minimal surface $X$ has a nef canonical line bundle.
If the tangent bundle of $X$ is strongly psef, which implies in particular that $-K_X$ is psef, $c_1(X)$ hence vanishes.
If the surface is Kähler, by the main theorem of \cite{Wu20}, the tangent bundle is in fact nef.
Thus we have a complete list by the work of \cite{DPS94}.

If the surface is not K\"ahler, then we know that $X$ is either a Kodaira surface or a properly elliptic surface by the classification of minimal surfaces.
By \cite{DPS94}, the first case gives an example of a surface with a strongly psef (in fact nef) tangent bundle.
Let $X$ be a properly elliptic surface with fibration $\pi: X \to C$.
We prove by contradiction that its tangent bundle is not strongly psef.
Note that the genus of $C$ is 1.
Otherwise, $H^0(C,\Omega^1_C)$ has a section with zeros which pulls back to a section with zeros on $X$.
Contradiction with the fact that $T_X$ is strongly psef.
By theorem 15.4 Chap. III of \cite{BPV}, there must exist a singular fibre for $\pi$.
Otherwise, $\pi$ is locally trivial and we have an exact sequence of vector bundles
$$0 \to T_{X/C} \to T_X \to \pi^* T_C \to 0.$$
Since $c_1(T_{X/C})=c_1(X)=0$, $T_X$ is nef.
By the result of \cite{DPS94}, $X$ should be a Kodaira surface.
Let $t \in C$ be a critical point of $X$.
Then $\pi$ is not a submersion near $t$ and there exists $x$ such that $d\pi(x)=0$.
Let $dz$ be a global section of $H^0(C, \Omega^1_C)$.
Then $\pi^* dz$ has a zero at $x$, which contradicts the fact that $T_X$ is strongly psef.

Now consider the case $\kappa(X)=- \infty$.
By the above discussion, it is enough to consider the non-projective case.
By the classification of minimal surface, it is enough to consider the surface of class $VII_0$, which is classified by the work of \cite{Bo82} and \cite{LYZ90}.
There are three cases: torus, Hopf surface or Inoue surface.
It is checked in \cite{DPS94}, the first two cases give surfaces with strongly psef (in fact nef)  tangent bundle.

For an Inoue surface $X$, it is checked in \cite{DPS94} that its tangent bundle is not nef.
The same argument works for the strongly psef case. We sketch the proof and refer to \cite{DPS94} for more details.
The Inoue surface associated with a matrix in $SL(3,\Z)$ and the Inoue surface of type $S _{N,p,q,r;t}^+$ has a quotient line bundle $Q$ of the tangent bundle, which is semi negative but not flat.
In particular,
$\int_X c_1(Q ) \wedge \omega <0$
where $\omega$ is a Gauduchon metric on $X$.
However, a quotient bundle of a strongly psef vector bundle is also strongly psef.
Contradiction.
The Inoue surface of type $S _{N,p,q,r;t}^-$
has a double \'etale cover of type $S _{N,p,q,r;t}^+$.
Its tangent bundle is not strongly psef since strongly psefness is preserved under finite \'etale cover.
}
\end{myex}
Note that any compact surface with a strongly tangent bundle has to be blow-up of a minimal compact surface with a strongly tangent bundle by the following easy remark.
\begin{myrem}
{\em 
Let $\pi: \tilde{X} \to X$ be a bimeromorphic morphism between smooth compact manifolds.
If $T_{\tilde{X}}$ is strongly psef, then $T_X$ is also strongly psef.
The reason is as follows.
The image of $T_{\tilde{X}} \to \pi^* T_X$ is a torsion-free sheaf which is strongly psef by (1) Proposition 4 of \cite{Wu20}.
On the other hand, the inclusion of the image into $\pi^* T_X$ is generically isomorphism since $\pi$ is bimeromorphic.
Thus by Proposition 5 of \cite{Wu20}, $\pi^* T_X$ is strongly psef.
By theorem 1 of \cite{Wu20}, $T_X$ is hence strongly psef.
}
\end{myrem}
\begin{myrem}
{\em 
By the fundamental work of Michel Brion \cite{Bri10}, we can show that the conjectural structure of the Albanese morphism of a compact projective manifold with a strongly psef tangent bundle holds for a quasi-homogeneous projective manifold $X$ with a connected algebraic group action $G$.
First of all, the Albanese morphism of $X$ can be written as
$$\alpha: X \to G/H$$
where $H$ is a linear subgroup such that $G/H$ is an abelian variety.
($H$ is unique by Chevalley's structure theorem.)
Moreover, Brion showed that one has a $G-$equivariant isomorphism $X \cong G \times^H F$ where the right-hand side denotes the homogeneous fibre bundle over
$G/H$ associated with the scheme-theoretic fibre $F$ of $\alpha$ at the base point, which is a normal variety.
Since $X$ is smooth, $F$ is also smooth.
This shows in particular that the Albanese morphism is locally trivial.

By proposition 3.12 (iii) of \cite{DPS94}, the fibre $F$ satisfies $\tilde{q}(F)=0$.
It is conjectured in \cite{HIM} that a projective manifold $F$ with a strongly psef tangent bundle and $\tilde{q}(F)=0$ has a big anticanonical line bundle.
It can be shown that this conjecture holds in the case of projective quasi-homogeneous manifolds.
With the notations above taking $F$ as $X$, $H=G$, which means that $F$ is a projective manifold, quasi-homogeneous under a linear algebraic
group.
Then by the work of \cite{FZ13}, $-K_F$ is big.

It is communicated to the author by Michel Brion that the stronger version of structure conjecture for the compact Kähler manifolds with a strongly psef tangent bundle holds for a smooth projective quasi-homogeneous manifold $X$ with the action of a connected algebraic group $G$.
The reason is as follows.
With the same notation above, we can assume $H$ to be connected.
Otherwise, consider $Y$ after the base change
$$
\begin{tikzcd}
Y \arrow[r] \arrow[d] & G/H^0 \arrow[d] \\
X \arrow[r]           & G/H            
\end{tikzcd}
$$
where $H^0$ is the connected component of the neutral element.
Since $G/H^0 \to G/H$ is finite \'etale, $Y$ is still smooth projective, and one has that the universal cover of $X$ is the same as the universal cover of $Y$.
By the affinisation theorem, a distinguished subgroup $K$ exists such that $G/K$ is a linear algebraic group.
$K$ is commutative and is central in $G$, which implies that $HK$ is a distinguished subgroup of $G$.
$G/HK$ as a linear algebraic group $G/K$'s quotient is also a linear algebraic group.
On the other hand, $G/HK$ as a quotient of the abelian variety $G/H$ is also an abelian variety.
In particular, $G/HK$ has to be trivial.
In other words, $G=HK$.
In particular, one has surjection
$$K \to K/K \cap H \cong HK /H \cong G/H.$$
Since $K$ is commutative and connected, the universal cover $\tilde{K}$ of $K$ is isomorphic to $\C^r$ for some $r \in \N$ with $\pi: \tilde{K} \to K$.
As above, considering some finite \'etale cover, we can assume that $H \cap K$ is connected.
Thus $\pi^{-1}(K \cap H)$ is a connected closed subgroup of $\C^r$ which is hence isomorphic $\C^s$ for some $s \leq r$.
There exists an isomorphism $\C^r  \cong \C^s \times \C^{r-s}$ as algebraic group.
Thus we have
$$\tilde{X} \cong \tilde{K} \times^{\pi^{-1}(K \cap H)} F \cong \C^{r-s} \times F$$
since $X \cong G \times^H F \cong K \times^{H \cap K} F.$ 
}
\end{myrem}
\begin{myrem}
{\em 
By the seminal work of \cite{Fuj78}, we can generalise the results in the previous remark from the projective case to the compact K\"ahler case.
More precisely, as in \cite{Fuj78}, we call a compact K\"ahler manifold $X$ almost homogeneous if $G:=\Aut_0(X)$ has an open orbit in $X$.
Here $\Aut(X)$ is the automorphism group of $X$ and  $\Aut_0(X)$ is the identity connected component of $\Aut(X)$. 
It is equivalent to our previous definition with more structures on $\Aut_0(X)$.
$\Aut(X)$ has a natural complex Lie group structure and 
$\Aut_0(X)$ also has a natural meromorphic group structure. 
Furthermore, there exists a unique meromorphic subgroup $L(X)$ of $\Aut_0(X)$, which is meromorphically isomorphic to a linear algebraic group and such that the quotient denoted by $T(X)$ is a complex torus.
This can be seen as a version of Chevalley's structure theorem for meromorphic groups in the sense of Fujiki.

Applying Lemma 6.1 in \cite{Fuj78} to $X$ a compact K\"ahler almost homogeneous manifold implies that 
the Albanese morphism of $X$ is a locally trivial submersion.
In fact, 
the results of Fujiki imply that
the natural morphism $G \to \Alb(X) $ is surjective.
In other words, the action of $G$ on $\Alb(X)$ is transitive.
Let $x \in X$ be a point such that $G \cdot x \simeq G/I$ is the open dense orbit in $X$.
Let $H$ be the isotopy group of $G$ at $\alpha(x)$ (such that $\Alb(X) \simeq G/H$).
Since $\alpha$ is $G-$equivariant, $I \subset H$.
Let $F$ be the fibre of $\alpha$ at $\alpha(x)$ (i.e. the identity element in $G/H$ under the identification).
Then we have a morphism $G \times F \to X$ by sending $(g,x) $ to $g \cdot x$ which is surjective since the preimage of $x$ can be taken as $(g, g^{-1}\cdot x)$ where $g \in G$ such that $\alpha(x)=gH$.
The group $H$ acts on $G \times F$ by $h \cdot (g,x):=(g h^{-1}, h \cdot x)$.
Thus we have a commutative diagram
$$
\begin{tikzcd}
G \times^H F \arrow[r, "\simeq"] \arrow[d] & X \arrow[d, "\alpha"] \\
G/H \arrow[r, "\simeq"]           & \Alb(X)           
\end{tikzcd}
$$
which implies local isotriviality.

The fact that $-K_F$ is big comes from the following arguments.
Recall in general that the dimension of $T(X)$ is always less than the dimension of $\Alb(X)$ even without the quasi-homogeneous condition (Theorem 5.5 of \cite{Fuj78}).
Since $\tilde{q}(F)=0$, up to a finite \'etale cover, $T(F)$ is a point.
In particular, up to a finite \'etale cover, $\Aut_0(F)$ is an algebraic group.
Then the work of \cite{FZ13} concludes that $-K_F$ is big up to a finite \'etale cover, which implies the bigness of $-K_F$.

The result on the universal cover follows from the following arguments.
In our case, $T(X) \to \Alb(X)$ is a finite \'etale cover by Theorem 5.5 of \cite{Fuj78}.
Thus up to a finite \'etale cover, to consider the universal cover of $X$, we can assume that $G/H \simeq T(X)$.
More precisely, we consider the following base change
$$
\begin{tikzcd}
G \times^{L(X)} F \arrow[r] \arrow[d] & G/L(X) \simeq T(X) \arrow[d] \\
G \times^H F \arrow[r]           & G/H.            
\end{tikzcd}
$$ 
In particular, we can assume that $H$ is a linear algebraic group without loss of generality.
The above arguments of Brion work identically in the compact K\"ahler case with slight modification.
Instead of the affinisation of the algebraic group, we need Theorem 3.5 of \cite{Fuj78} for $G$.
More precisely, there exists a unique pair $(G_1,\rho)$ up to isomorphism, which consists of a linear algebraic group $G_1$ and a surjective meromorphic homomorphism $\rho: G \to G_1$ such that the kernel of $\rho $ is connected and is contained in the centre of $G$.
In particular, $K$ can be taken to be the kernel of $\rho$.
Of course, moreover, all abelian varieties have to be changed to be complex tori.
}
\end{myrem}
\begin{myrem}
{\em
We can even generalise the results in the previous remarks to compact manifold of class $\cC$ (i.e. compact manifold bimeromorphic to K\"ahler manifold).
In fact, the results of Fujiki on meromorphic groups used in the previous remark hold for a compact manifold $X$ of class $\cC$ as already stated in \cite{Fuj78}.

The non-trivial point is that we cannot apply Proposition 3.12 of \cite{DPS94} on the augmented irregularity for a compact manifold of class $\cC$.
The arguments need to be modified to study Albanese morphism's fibre.

The same arguments show that Albanese morphism is a locally trivial submersion.
As observed in Theorem 1 of \cite{JM22}, there exists an $\Aut_0(X)-$equivariant modification $\pi: \tilde{X} \to X$ obtained by equivariant blow-ups of smooth centres such that $\tilde{X}$ is K\"ahler.
Since every orbit is open in its closure, the induced $\Aut_0(X)-$action on $\tilde{X}$ has an open dense orbit.
In other words, $\tilde{X}$ is $\Aut_0(X)-$quasi-homogeneous.
Note that by Lemma 2.4 (4) of \cite{Fuj78}, the image of $\Aut_0(X) \to \Aut_{0}(\tilde{X})$ has a meromorphic structure.
Note also that the Albanese torus of $X$ and $\tilde{X}$ are isomorphic.

By the previous remark, the fibre $\tilde{F}$ of the Albanese morphism of $\tilde{X}$ has big anticanonical divisor with $\tilde{q}(\tilde{F})=0$.
The restricting of $\pi$ on $\tilde{F}$ gives a modification of the fibre $F$ of the Albanese morphism of $X$.
Thus $-K_{\tilde{F}}$ is big implies that $-K_F$ is big.
On the other hand, $F$ is rationally connected as $\tilde{F}$.
Thus the fundamental group of $F$ and $\tilde{F}$ is trivial, which implies that
$\tilde{q}(F)=q(F)=q(\tilde{F})=0$.
The rest arguments on the universal cover of $X$ are the same.

Note that the example of Hopf surface shows that Albanese morphism's fibre of an arbitrary homogeneous compact complex manifold has not necessarily a big anticanonical line bundle.
}
\end{myrem}

\begin{myex}
{\em 
This example shows that a projective manifold with a psef anticanonical line bundle may have a fundamental group with exponential growth.

Take $Y$ to be a complex curve of genus bigger than 2. 
Fix $A$ an ample divisor on $Y$.
Define $E= A^{\otimes p} \oplus A^{ \otimes -q }$ where $p, q \in \N$ are suitably chosen.
Denote $X = \P(E)$ with $\pi: X \to Y$. 
For well-chosen $p,q$, the anticanonical line bundle $-K_X$ is big but not nef in codimension 1 (for more details cf. \cite{Wu19}). 
Since $\pi$ is a fibration with projective spaces as fibres,
the long exact sequence associated to the fibration shows
$$\pi_1(\P^1)=\{e\} \to \pi_1(X) \xrightarrow{\simeq} \pi_1(Y) \to \pi_0(\P^1)=\{e\}.$$
In particular, $\pi_1(X)$ has a non-abelian free subgroup.
}
\end{myex}
\section{Characterisation of projective space}
In this section, following the idea of Iwai \cite{Iwe}, we give new proof of the characterisation of projective space.

First, we introduce the notion of the strong bigness of a vector bundle.
\begin{mydef}
Let $E$ be a vector bundle over a compact K\"ahler manifold $X$. 
$E$ is said to be strongly big if there exists a K\"ahler current $T$ in the class of $c_1(\cO_{\P(E)}(1))$ such that the projection of the singular part $E_{+}(T)=\{\nu(T,x) >0\}$ does not cover $X$.
Here $\nu(T,x)$ is the Lelong number of $T$ at $x \in \P(E)$.
\end{mydef}
\begin{myrem}
By Demailly's regularisation theorem,
in the definition a K\"ahler current can be changed by a K\"ahler current with analytic singularities.
\end{myrem}
When the manifold is projective, it is the same as the so-called big vector bundle defined in \cite{Iwe}.
\begin{mylem}
Let $E$ be a holomorphic vector bundle over a projective manifold $X$.
$E$ is strongly big iff there exist an ample line bundle $A$ and $k \in \N >0$ such that $\Sym^k(E) \otimes A^{-1}$ is positive in the sense of Nakayama at some point $x$.
\end{mylem}
\begin{proof}
The proof is similar to the proof of Proposition 7.2 of \cite{BDPP}.
For the convenience of the reader, we give the outline of the proof.

If there exists $k \in \N >0$ such that $\Sym^k(E) \otimes A^{-1}$ is positive in the sense of Nakayama at some point $x$.
Then there exist $p,m \geq 2$ such that the global sections of 
$\Sym^p((\Sym^m (\Sym^k(E) \otimes A^{-1})) \otimes A)$
is globally generated at $x$.
In particular, there exists a non-empty Zariski open set $U_{m,p,k}$ such that these global sections globally generate any point in $U_{m,p,k}$.
They can be used to define a singular hermitian metric $h_{m,p,k}$ on $\cO_{\P(E)}(1)$ which has poles contained in $\pi^{-1}(X \setminus U_{m,p,k})$ and whose curvature form satisfies
$$\Theta_{h_{m,p,k}}(\cO(1)) \geq \frac{m-1}{mk} \pi^* \Theta(A).$$
Conversely, without loss of generality, we can assume that $T$ has analytic singularities.
Up to a multiple, identify $T$ as a singular metric on $\cO_{\P(E)}(1)$.
In particular, for $k$ large enough, we have a curvature estimate
$$\Theta_{}(\cO(k) \otimes \pi^*{A}^{-1}) \geq \frac{1}{2} \pi^* \Theta(A)$$
induced from this singular metric smooth on a non-empty Zariski open set.
After multiplying $\Theta(A)$ by a sufficiently large integer $p$ to compensate
the curvature of $-K_X$,
by $L^2-$estimate, we can construct global sections of $\Sym^p(\Sym^k E \otimes A^{-1})$.
The sections possibly still have to vanish along the poles of the metric, but they are unrestricted on the fibres of $\P(\Sym^k E) \to X$ which do not meet the singular part of the metric.
In particular, we prove that $\Sym^p(\Sym^k E \otimes A^{-1})$ is generically globally generated, which implies that $S^k E \otimes A^{-1}$ is positive in the sense of Nakayama.
\end{proof}
In fact, the definition of a strongly big vector bundle implies the projectivity of the manifold.
In particular, our definition is precisely equivalent to that of Iwai.
(Of course, it is interesting to know whether the same holds on a singular space.
Although the $L^2-$method is insensitive to the singularity, we do not have the regularisation to control the analyticity of the singular part.)
\begin{mylem}
If $X$ admits a strongly big vector bundle $E$ of rank $r$, $X$ is projective.
\end{mylem}
\begin{proof}
Since $X$ is assumed to be K\"ahler, it is enough to show that $X$ is Moishezon by the existence of a big line bundle.
By the construction of Serge current in \cite{Wu20}, we can define a closed positive current $\pi_*(T^r)$ in the class of $\pi_*(c_1(\cO_{\P(E)}(1))^r)=c_1(\det(E))$.

If $T \geq \pi^* \omega$ where $\omega$ is a K\"ahler form on $X$,
$$\pi_*(T^r)\geq \pi_*(T^{r-1} \wedge \pi^* \omega)=\pi_*(c_1(\cO_{\P(E)}(1), h_\infty)^{r-1} \wedge\pi^* \omega) = \omega$$
by the construction of Serge current where $h_\infty$ is any smooth metric on $\cO_{\P(E)}(1)$.
In other words, $\pi_*(T^r)$ is a K\"ahler current in $c_1(\det(E))$ which implies that $\det(E)$ is a big line bundle over $X$.
\end{proof}
In fact, we have the following more general fact.
\begin{myprop}
Let $\pi: X \to Y$ be a submersion of relative dimension $r$ between compact complex manifolds.
If $X$ is projective, so is $Y$.
\end{myprop}
\begin{proof}
Let $A$ be an ample line bundle over $X$ with smooth metric $h$ such that $c_1(A,h)$ is a K\"ahler form on $X$.
Consider the rational Bott-Chern cohomology group defined by Schweitzer \cite{Sch}.
We have the following commuting diagram
$$
\begin{tikzcd}
{H^{1,1}_{BC}(X, \Q)} \arrow[r] \arrow[d] & {H^{1,1}_{BC}(Y, \Q)} \arrow[d] \\
{H^{1,1}_{BC}(X, \R)} \arrow[r]           & {H^{1,1}_{BC}(Y, \R)}          
\end{tikzcd}
$$
where the vertical arrows are the canonical maps and the horizontal map is sending $\alpha$ to $\pi_*(\alpha^{r+1})$ in the corresponding cohomology groups.
Since $\pi$ is a submersion,
$\pi_*(c_1(A,h)^{r+1})$ is a K\"ahler form on $Y$ in the class $H^{1,1}_{BC}(Y, \R)$.
Consider the commuting diagram
$$
\begin{tikzcd}
{H^{1,1}_{BC}(Y, \Q)} \arrow[r] \arrow[d] & {H^{2}(Y, \Q)} \arrow[d] \\
{H^{1,1}_{BC}(Y, \R)} \arrow[r]           & {H^{2}(Y, \R)}          
\end{tikzcd}
$$
which implies that $\pi_*(c_1(A,h)^{r+1})$ is contained in $H^{1,1}_{BC}(Y, \R) \cap H^2(Y, \Q)$.
Thus $Y$ is projective by Kodaira's criterion.  
\end{proof}
\begin{myrem}
{\em 
The condition that $E$ is strongly big is stronger than the condition that $\cO_{\P(E)}(1)$ is a big line bundle and that $E$ is a strongly psef vector bundle.

A typical example is the following.
Consider $E = \cO_X \oplus A$ with an ample line bundle $A$ over a projective manifold $X$.
As a direct sum of strongly psef vector bundles, $E$ is strongly psef.
Classically, once one component in the direct sum $E'=\oplus L_i$ is big, $\cO_{\P(E')}(1)$ is a big line bundle.

However, $E$ is not strongly big.
Otherwise, there exists $k >0$ such that $\Sym^k E \otimes A^{-1}$ is positive in the sense of Nakayama.
On the other hand,
$$\Sym^k E \otimes A^{-1}=\oplus_{i=0}^k A^{i-1}$$
which contains a strictly negative component $A^{-1}$.
}
\end{myrem}
Now, we give another proof of Corollary 1.3 of \cite{Iwe} using Segre current.
This result was also proved by Fulger and Murayama \cite{FM} by using Seshadri constants of vector bundles.

Let us first recall the essential ingredient of the proof due to Corollary 0.4 of \cite{CMSB}.
\begin{mythm}(Length condition on rational curves with base point)

Let $X$ be a  complex projective manifold of dimension $n$ and $x_0 \in X$  be a general point.
$X$ is biholomorphic to $\P^n$ if and only if
$X$ is uniruled and $( C, -K_X) \geq n+ 1$ for every rational curve $C$ containing the prescribed general point $x_0$.
\end{mythm}
\begin{mycor}
Let $X$ be a compact K\"ahler manifold of dimension $n$ with a strongly big tangent bundle.
Then $X$ is biholomorphic to $\P^n$.
\end{mycor}
\begin{proof}
Using the K\"ahler current constructed in the above lemma 6. 
If we start from a K\"ahler current $T$ with analytic singularities on $\P(T_X)$,
the singular part of $\pi_*(T^n)$ is contained in a proper analytic set of $X$.

Since $-K_X$ is a big line bundle and $X$ is projective, by \cite{BDPP}, $X$ is uniruled.
For any rational curve $C$ not contained in the singular part of $\pi_*(T^n)$, the restriction of $\pi_*(T^n)$ to $C$ is well defined.
Since $C$ is a rational curve, the restriction of $T_X$ on $C$ splits
$$T_X|_C \cong \oplus_{i=1}^n \cO_C(a_i)$$
with $a_1 \geq \cdots \geq a_n$ and $a_1 \geq 2$.
By the above theorem 3, it is enough to show that $a_n \geq 1$.

$T$ up to a multiple uniquely determines a singular metric on $\cO_{\P(T_X)}(1)$.
The restriction of the singular metric to $\P(T_X|_C)$ is well defined for the choice of $C$.
The restriction of $\cO_{\P(T_X|_C)}(1)$ to $\P(\cO_C(a_n)) \cong X$ is $\cO(a_n)$ under the identification.
Thus we have
$$a_n=\int_{\P(\cO_C(a_n))} T >0$$
since $T$ is a K\"ahler current on $\P(T_X)$.
Since $a_n$ is a natural number, $a_n  \geq 1$.
\end{proof}
Using similar arguments and following the observation in \cite{CS95}, we give the following estimate of the length of the rational curve passing a general point.
Recall that the length of a projective uniruled manifold $X$ (resp. local length of $X$ at $x \in X$) is defined to be the infimum of $(-K_X \cdot C)$ along all rational curves $C$ (resp. all rational curve containing $x$).
The local length is lower semicontinuous for $x$; thus, the length of the rational curve passing a general point is meaningful.
\begin{mycor}
Let $X$ be a projective manifold of dimension $n \geq 4$ such that $\wedge^2 T_X$ is strongly big and that $X$ is not a projective space.
Then the length of the rational curve passing a general point is equal to $n$.
\end{mycor}
\begin{proof}

We follow the proof of Corollary 1 and use its notations.
In this case,
$$\wedge^2 T_X|_C \cong \oplus_{i \neq j}^n \cO_C(a_i+a_j).$$
Using the Segre current as in the proof, we know that $a_{n-1}+a_n > 0$ for a rational curve $C$ passing the general point.

Since $a_{n-1}+a_n > 0$, $a_{n-1} >0$.
Thus we have that
$$a_1+a_2 +\cdots +a_{n-2}+(a_{n-1}+a_n) \geq 2+ (n-3)+1=n.$$
The case of length at least $n+1$ by Theorem 3 would imply that $X$ is projective space.
Thus it is enough to consider the equality case in the above inequality (i.e. $a_1=2, a_2= \cdots=a_{n-2}= a_{n+1}+a_n=1$.)

If $n \geq 4$, $1=a_{n-2} \geq a_{n-1} >0$.
Thus $a_{n-1}=1$, and the length of the rational curve passing a general point is equal to $n$.
\end{proof}
\begin{myrem}{\em 
The following statement appeared in a paper of Miyaoka in \cite[Theorem 0.1]{Miy}.
\stateparagraph
Let $X$ be a  smooth Fano variety of dimension $n \geq 3$ defined over an algebraically closed field $k$ of characteristic zero. Then the following three conditions are equivalent: 
\begin{enumerate}
\item $X$ is isomorphic to a  smooth hyperquadric $Q_n \subset \P^{n+1}$.
\item The length of $X$ is $n$.
\item The length of the rational curve passing a general point is equal to $n$, and the Picard number of $X$ is 1.
\end{enumerate}
As pointed out in \cite{DH17} (cf. Remark 5.2 in \cite{DH17}), the proof there is incomplete.
In \cite{DH17}, they prove the equivalence of (1) and (2) using different methods.

If one can show that (3) implies (1), then our argument concludes that if $X$ be a Fano manifold with Picard number 1 of dimension $n \geq 4$ such that $\wedge^2 T_X$ is strongly big and that $X$ is not a projective space,
then $X$ is biholomorphic to a quadric hypersurface.
This characterises the ``second most simple" variety in terms of ``singular" positivity on the tangent bundle.
Recall that in \cite{CS95}, it is proved that if $X$ be a manifold of dimension $n \geq 3$ such that $\wedge^2 T_X$ is ample and that $X$ is not a projective space,
Then $X$ is biholomorphic to a quadric hypersurface.

Notice that the condition of the Picard number is necessary for the conclusion.
As observed in \cite[Remark 4.2]{Miy},
there exists a Fano manifold $X$ of dimension $n$, which is not a quadric hypersurface such that the length of the rational curve passing a general point is equal to $n$.
In fact, let $H \subset \P^n$ be a hyperplane and $A \subset H \subset \P^n$ be a projective manifold $A$ of dimension $n-2$ and degree $3\leq a \leq n$.
It is enough to take $X$ as the blow-up of $\P^n$ along $A$.
}
\end{myrem}

\section{Characterisation of nef vector bundle}
In the definition of the pseudoeffective vector bundle, an additional condition is made on the approximate singular metrics on the tautological line bundle.
One may ask whether the additional condition could be made directly on some positive singular metric on the tautological line bundle.
For example, let $h$ be a positive singular metric on $\cO_{\P(E)}(1)$, could we consider the condition that the projection of singular set $\{h = \infty\}$ in $X$ is a (complete) pluripolar set?
However, this kind of condition does not behave well functorially as shown below.
In particular, the image of a (complete) pluripolar set under a proper morphism is not necessarily (complete) pluripolar as shown in the following example 4.
Notice that by Remmert's proper mapping theorem, the image of a closed analytic set under a proper morphism is always (closed) analytic.
The projection of singular set $\{h = \infty\}$ in $X$ could be very bad which forbids the Skoda-El Mir type of extension theorem.

Recall that a subset $F$ of a complex manifold $X$ is said to be pluripolar if for
each point $z \in F$ there is a plurisubharmonic function $u$, $u \neq - \infty$, defined in an open
neighborhood $U$ of $z$, such that $F \cap U \subset \{u = -\infty\}$. 
A subset $F$ of a complex manifold $X$ will be called to be a complete pluripolar set if there is a
non-constant quasi-plurisubharmonic function $u$ defined on $X$ such that $F = \{u =-\infty \}$.
By the results of \cite{Col90}, this definition is equivalent to the condition that for
each point $z \in F$ there is a plurisubharmonic function $u$, $u \neq - \infty$, defined in an open
neighbourhood $U$ of $z$, such that $F \cap U = \{u = -\infty\}$ (i.e. complete locally pluripolar) if $F$ is closed.
In fact, they prove the following lemma.
\begin{mylem}(Lemma 1, \cite{Col90})
Let $X$ be a complex space and $A \subset X$ a closed complete locally pluripolar set.
Then $A$ can be defined by $\{U_i, \varphi_i\}_{i \in \N}$ with $U_i$ open and relatively compact in $X$, $\{U_i\}_{i \in \N}$
locally finite.
There exist $U''_i \Subset U'_i \Subset U_i \Subset X$ such that
$X = \cup_i U''_i$, $\varphi_i: U_i \to [-\infty, \infty[$ plurisubharmonic functions,
$$A \cap U_i=\{\varphi_i=-\infty\},$$
$\exp(\varphi_i)$  continuous, $\varphi_i$ smooth outside $A$ and $\varphi_i-\varphi_j$ bounded on $U'_i \cap U'_j \setminus A$.
\end{mylem}
Since $\varphi_i-\varphi_j$ is bounded on $U'_i \cap U'_j \setminus A$,
we can choose $p_i \in C_c^\infty(U'_i)$ such that 
$$\varphi_i+p_i < \varphi_j+p_j$$
on $\d U'_i \cap U''_j \setminus A$ and $p_i \geq 0$.
In particular, 
$$\psi(z):=\max\{\varphi_i(z)+p_i(z); z \in U'_i\}$$
is a quasi psh function as the maximum of finite quasi psh functions locally.
Note that $A$ is the pole set of $\psi$.

 Notice that the image of a pluripolar set under a proper morphism is not necessarily pluripolar.
 For example, let $Z$ be a non-pluripolar set in $\P^1$.
 Consider $A:=Z \times \{0\} \subset \P^1 \times \P^1$.
 Since $A \subset \P^1 \times \{0\}$, $A$ is pluripolar.
 However, the image of $A$ under the projection of $\P^1 \times \P^1$ onto its first component is not pluripolar.
 The statement is still false if we change the condition ``pluripolar" to ``complete pluripolar"
 as shown in the following example recommended to the author by Demailly.
\begin{myex}
{\em 
In this example, we construct a complete pluripolar subset $A \subset \P^1 \times \P^1$ such that the image of $A$ under the projection onto the first component is $\Q \subset \C \subset \P^1$.
Notice that as countable union of pluripolar sets in $\C$, $\Q$ is pluripolar.
On the other side, $\Q$ is not complete pluripolar. Otherwises, there exists a K\"ahler form $\omega$ on $\P^1$ and a quasi-psh function $\varphi$ on $\P^1$ such that $\Q=\{\varphi=-\infty\}$ and $\omega + i \d \dbar \varphi \geq 0$ in the sense of currents.
In particular, over $\C$, $\omega|_\C=i \d \dbar \psi$ for some $\psi \in C^\infty(\C)$ and $\varphi+\psi$ is psh on $\C$ with pole set $\Q$. 
Thus, $\Q$ should be a $G_\delta$ subset in $\C$ (i.e. $\Q=\cap_{j \in \N} V_j$ for $V_j$ open in $\C$).
Notice that $\Q$ is dense in $\C$ as well as all the $V_j$'s.
Thus we have that 
$$\emptyset= \Q \cap (\C \setminus \Q)=\cap_{j \in \N} V_j \cap \cap_{q\in \Q} \C \setminus \{q\} $$
which contradicts the Baire category theorem.

Define 
$$A := \cup_{a \in \N^*, b\in \Z} \{z_{a,b}= (b/a,a)\} \subset \C^2 \subset \P^1 \times \P^1.$$
Define for any $z \in \C^2$,
$$\delta(z):= \inf_{z \neq z_{a,b}} \log|z-z_{a,b}|^2$$
which is not infinite on $\C^2 \setminus A$ implying that $A$ is discrete in $\C^2$.

Consider a sequence of $a_{a,b}>0$ such that $\sum_{a,b} a_{a,b} < \infty$
which would be chosen later.
Consider 
$$\varphi'(z):=\sum_{a,b} a_{a,b} (\log|z-z_{a,b}|^2-\log(1+|z|^2)).$$
Notice that for any $z \notin A$, $\varphi'(z) \geq \sum_{a,b} a_{a,b} (\delta(z)-\log(1+|z|^2)) \neq - \infty$
 and for any $z_{a,b}$, $\varphi'(z_{a,b})=-\infty$.
 In particular, $A \cap\C^2 =\{ \varphi'=-\infty\}$.
Consider
$$C_{a,b}=\sup_{z \in\C^2}( \log|z-z_{a,b}|^2-\log(1+|z|^2) )< \infty.$$
With suitable choice of $a_{a,b}>0$, we can assume that $\sum_{a,b} a_{a,b} C_{a,b} < \infty$.
In particular, $\varphi'$ is uniformly bounded from above on $\C^2$ which extends to be a quasi-psh function on $\P^1 \times \P^1$ with pole set $A$.
 }
\end{myex}

It was proven in  \cite{Pau98a}, \cite{Pau98b} that a line bundle over a compact complex manifold is nef if and only if its restriction to all irreducible analytic subsets is psef.
The precise meaning of a nef line bundle over an irreducible complex space will be discussed in the following.
In this section, we generalise the result of P\u{a}un to the higher rank case.

Note that there exist different notions of psef vector bundles in the higher rank case.
The weakest one is the following: a vector bundle $E$ is called psef if $\cO_{\P(E)}(1)$ is a psef line bundle.
Compared to the result of P\u{a}un, the analogous result would be false with this too weak definition, as shown in the next example.
\begin{myex}
{\em
Consider $E=A \oplus A^{-1}$ with $A$ a very ample line bundle over a curve $C$.
Then 
$$H^0(\P(E), \cO_{\P(E)}(1))=H^0(C, A) \neq 0$$
which implies that $\cO_{\P(E)}(1)$ is a psef line bundle.

However, $E$ is not a nef vector bundle since its quotient can be non-nef.
The analytic subset of $C$ is either $C$ itself or the points. 
In the two cases, the restriction of $E$ to these analytic sets is psef in the weak sense.
In particular, we can not hope that if the restriction to a vector bundle to any analytic set is psef in the weak sense, the vector bundle has to be nef.
}
\end{myex}
Let us recall some basic definitions of the psh functions in a complex space.
We recommend the article \cite{Dem85} for further information and reference.

Recall that the definitions of test functions and currents are local in nature.
To define them on a complex space $X$,
we can identify $X$ as a closed analytic subset of an open set $\Omega \subset \C^N$.
Then smooth forms on $X$ are defined as the image of the restriction map of smooth forms on $\Omega$ on the regular part of $X$.
The topology is induced by quotient topology.
The corresponding dual space is then defined as the space of currents on $X$.
\begin{mydef}(D\'efinition 1.9 \cite{Dem85})
A locally integrable function $V$ over a complex space $Z$ is called weakly psh (resp. weakly quasi psh) if $V$ is locally bounded from above and $i \d \dbar V \geq 0$ in the sense of currents (resp. $i \d \dbar V \geq \alpha$ in the sense of currents with $\alpha$ a smooth form on $Z$).

\end{mydef}
When the complex space $Z$ is smooth, the condition that $V$ is locally bounded from above follows from the condition that $V$ is locally integrable and $i \d \dbar V \geq 0$ (resp. $i \d \dbar V \geq \alpha$ in the sense of currents with $\alpha$ a smooth form on $Z$).
However, this is not always the case in the singular setting.

We have also the definition of the psh function over a complex space.
\begin{mydef}
(D\'efinition 1.9 \cite{Dem85})
Let $V: Z \to [-\infty, \infty[$ be a not identically infinite function over any open set of $Z$.
Then $V$ is called psh (resp. quasi-psh) if for any local embedding $j: Z \hookrightarrow \Omega \subset \C^N$, $V$ is the local restriction of a psh (resp. quasi-psh) function on $\Omega$.
\end{mydef}
We have the following equivalent definition of psh functions due to J.E. Fornaess and R. Narasimhan.
\begin{mythm}
(Theorem 5.10 \cite{FN})

$V$ is a psh function over a complex space $X$ if and only if

(1) $V$ is upper semi-continuous.

(2) For any holomorphic map $f: \Delta \to X$ from the unit disc $\Delta$,
either $V \circ f$ is subharmonic or $V \circ f$ is identically infinite. 
\end{mythm}
As a direct consequence, the pullback of a quasi-psh function between complex spaces is still quasi-psh.

Of course, the definitions of psh function and weakly psh function over a complex manifold coincide.
However, the definitions of psh and weakly psh function are different in general for a complex space.
\begin{myex}
{\em
Let $Z \subset \C^2$ be a complex space defined by $z_1z_2=0$. 
Consider a function $V$ which is identically equal to 1 on $\{z_1=0\}$ and identically equal to 0 otherwise.
Then $V$ can not be the restriction of some psh function on any open neighbourhood of 0 in $\C^2$.
Otherwise, the restriction of such a function on $\{z_2=0\}$  should be identically 0, which contradicts its value at the origin.

We claim, however, $V$ is a weakly psh function on $Z$.
In fact $i \d \dbar V=0$.
It is enough to consider the open neighbourhood of the origin.
Let $g$ be a test function of $Z$ near the origin.
By applying Stokes lemma $\langle i \d \dbar V, g \rangle=\langle V, i \d \dbar g \rangle=\int_{\{z_1=0\} } i \d \dbar g =0$ for any $g$.
}
\end{myex}
Now we can define the notion of nef/ strongly psef torsion-free sheaf over an irreducible compact K\"ahler complex space.
\begin{mydef}[Nef/ Strongly psef torsion-free sheaf]\strut\\
Assume that $\cF$ is a torsion-free sheaf over an irreducible compact K\"ahler complex space $X$.
We say that $\cF$ is nef $($resp.\ strongly psef$\,)$ if there exists some modification
$\pi: \tilde{X} \to X$ such that $\pi^* \cF/\tors$ is a nef
$($resp.\ strongly psef$\,)$ vector bundle over $\tilde{X}$ a smooth manifold where $\tors$ means the torsion part.
\end{mydef}
\begin{myrem}{\rm
By the work of \cite{Ros}, \cite{GR}, \cite{Rie}, for any torsion-free sheaf $\cF$ over a compact irreducible complex space, there exists a modification $\pi: \tilde{X} \to X$ such that $\pi^* \cF/\tors$ is a locally free sheaf (i.e.\ a vector bundle) over a complex manifold $\tilde{X}$.
In the above definition, we say that $\cF$ is nef or strongly psef if $\pi^* \cF/\tors$ is nef or strongly psef.

As shown in Remark 3 of \cite{Wu20}, in the above definition, it is the same to say that 
$\cF$ is nef $($resp.\ strongly psef$\,)$ if for {\it every} modification
$\pi: \tilde{X} \to X$ such that $\pi^* \cF/\tors$ is locally free over $\tilde{X}$ a smooth manifold, $\pi^* \cF/ \tors$  is a nef
$($resp.\ strongly psef$\,)$ vector bundle.
}
\end{myrem}
By the filtration property of modifications, the pullback of a strongly psef vector bundle is always strongly psef under a modification, even for the morphism between irreducible complex spaces.
More generally, we have the following result.
\begin{myprop}
Let $f: X \to Y$ be a surjective holomorphic map between irreducible compact K\"ahler complex spaces.
Let $\cF$ be a torsion-free sheaf over $Y$.
Then $\cF$ is nef (resp. strongly psef) if and only if $f^* \cF/\tors$ is nef (resp. strongly psef).
\end{myprop}
\begin{proof}
There exist modifications of $X$ and $Y$ such that we have the following diagram
$$
\begin{tikzcd}
\tilde{X} \arrow[r, "\pi_X"] \arrow[d, "\tilde{f}"]
& X \arrow[d, "f"] \\
\tilde{Y} \arrow[r, "\pi_Y" ]
& Y
\end{tikzcd}
$$
where $\tilde{X}$ and $\tilde{Y}$ are smooth.
The existence of such modification is ensured by the work of Hironaka \cite{Hir64}. 
More precisely, for the desingularisation $\tilde{Y}$ of $Y$, there exists a meromorphic map from $X$ to $\tilde{Y}$ since $f$ is surjective.
Taking the desingularisation of the graph of this meromorphic map defines $\tilde{X}$.
Since $\tilde{f}$ is proper, the image of $\tilde{f}$ is a closed analytic set containing a non-empty Zariski open set.
Thus $\tilde{f}$ is also surjective.

By theorem 1 of \cite{Wu20}, $\pi_Y^* \cF /\tors$ is nef (resp. strongly psef) if and only if $\tilde{f}^*(\pi_Y^* \cF /\tors)= \pi_X^*(f^* \cF/\tors) $ is nef (resp. strongly psef).
\end{proof}

Let us give an equivalent definition when $Z$ is an irreducible analytic subset of $X$ some compact K\"ahler manifold $X$.
\begin{mylem}
Let $Z$ be an irreducible analytic subset of $(X, \omega)$ some compact K\"ahler manifold $X$ and $E$ be a vector bundle over $X$.
Denote $\Theta(\cO_{\P(E)}(1))$ the Chern curvature of some smooth metric on $\cO_{\P(E)}(1)$.
Then $E|_Z$ is strongly psef (resp. nef) if and only if
for any $\varepsilon>0$ there exists a quasi-psh function $\varphi_\varepsilon$ on $\P(E|_Z)$ such that the projection of the pole set of $\varphi_\varepsilon$ is not $Z$ (resp. $\varphi_\varepsilon \in C^\infty(\P(E|_Z))$) such that
$$\Theta(\cO_{\P(E)}(1)) + i \d \dbar \varphi_\varepsilon \geq - \varepsilon\pi^* \omega$$
in the sense of currents on $\P(E|_Z)$ with $\pi: \P(E) \to X$.
\end{mylem}
\begin{proof}
Let $\sigma$ be a modification $\sigma: \tilde{Z} \to Z$ such that $\tilde{Z}$ is a K\"ahler manifold.
This kind of modification can be obtained as the composition of blows-up of the smooth centres of $X$ where $\tilde{Z}$ is the strict transform of $Z$ under the modification $\tilde{\sigma}: \tilde{X} \to X$ by the work of Hironaka.

If there exists the family of functions $\varphi_\varepsilon$ prescribed in the equivalent condition, we have
$$ \Theta(\cO_{\P(\tilde{\sigma}^* E)}(1)) + i \d \dbar \tilde{\sigma}^* \varphi_\varepsilon \geq - \varepsilon \tilde{\pi}^*\tilde{\sigma}^* {\omega} $$
in the sense of currents on $\P(\tilde{\sigma}^* E|_{\tilde{Z}})$.
Since there exists $C>0$ such that $\tilde{\sigma}^* \omega \leq C \tilde{\omega}$ where $\tilde{\omega}$ is any K\"ahler form on $\tilde{X}$,
$$ \Theta(\cO_{\P(\tilde{\sigma}^* E)}(1)) + i \d \dbar \tilde{\sigma}^* \varphi_\varepsilon \geq - C \varepsilon \tilde{\pi}^*\tilde{\omega} $$
in the sense of currents on $\P(\tilde{\sigma}^* E|_{\tilde{Z}})$.
In particular, $\sigma^* E|_Z$ is a strongly psef vector bundle over $\tilde{Z}$.
The proof of the nef case is similar.

Conversely, assuming that $\sigma^* E|_{\tilde{Z}}$ is a strongly psef (resp. nef), we want to prove that $E|_Z$ admits the Finsler metrics satisfying the curvature condition.
Let us first consider the strongly psef case.

By assumption that $\sigma^* E|_{\tilde{Z}}$ is a strongly psef, for any $\varepsilon>0$ there exist a quasi-psh function $\tilde{\varphi}_\varepsilon$ on $\tilde{Z}$ such that
$$\tilde{\sigma}^* \Theta(\cO_{\P(E)}(1)) + i \d \dbar \tilde{\varphi}_\varepsilon \geq - \varepsilon \tilde{\pi}^*\tilde{\omega} $$
in the sense of currents on $\P(\tilde{\sigma}^* E|_{\tilde{Z}})$ with $\tilde{\pi}: \P(\tilde{\sigma}^* E) \to \tilde{X}$.
Here we take the following choice of K\"ahler form on $\tilde{X}$.
Let $\Theta(E)$ be the Chern curvature of the exceptional divisor of $\tilde{\sigma}$. 
(Without loss of generality, we can assume the exceptional divisor is an SNC divisor by a possible further blow-up.)
Then $\tilde{\sigma}^* \omega-c \Theta(E)$ is a K\"ahler form on $\tilde{X}$ for $c >0$ small enough.
Since $\tilde{X}$ is smooth manifold, $\d \dbar-$lemma implies that
there exists a quasi-psh function $\psi_E$ on $\tilde{X}$ such that
$[E]=\Theta(E) +i \d \dbar \psi_E$
where $[E]$ is the positive current associated to the exceptional divisor.
In particular, we have that
$$\tilde{\sigma}^* \Theta(\cO_{\P(E)}(1)) + i \d \dbar (\tilde{\varphi}_\varepsilon+\varepsilon c\tilde{\pi}^* \psi_E) \geq -\varepsilon \tilde{\sigma}^*\pi^* {\omega}+\varepsilon c \tilde{\pi}^* [E] $$
in the sense of currents.

Now we push forward the function $\tilde{\varphi}_\varepsilon+\varepsilon c \tilde{\pi}^* \psi_E$ to a weakly quasi-psh function on $\P(E|_Z)$.
Then we modify the function to get a quasi-psh function on $\P(E|_Z)$.

Push forward the above current by $\tilde{\sigma}$ gives
$$ \Theta(\cO_{\P(E)}(1)) + i \d \dbar (\tilde{\sigma}_* \tilde{\varphi}_\varepsilon+\varepsilon c \tilde{\sigma}_* \tilde{\pi}^* \psi_E) \geq -\varepsilon \pi^* {\omega} $$
in the sense of currents over $\P(E|_Z)$.
In particular, $\tilde{\sigma}_* \tilde{\varphi}_\varepsilon+\varepsilon c \tilde{\sigma}_* \psi_E$ a weakly quasi-psh function over $\P(E|_Z)$.
It is bounded from above by construction.
To modify the weakly quasi-psh function to a psh function, we recall the th\'eor\`eme 1.10 of \cite{Dem85}.

{\it 
We denote $V_Y^*(a)$ the essential upper limit of a measurable function $V$ valued in $[-\infty,\infty[$ over a germ of complex space $(Y, a)$ at $a$.
A weakly psh function $V$ over a complex space $X$  is equal to almost everywhere to a psh function if and only if  near any point $a$, for any local irreducible component $(X_j, a)$ of $(X, a)$, we have
$$V_X^*(a)=V_{X_j}^*(a).$$
}

$ \Theta(\cO_{\P(E)}(1)) $ and $ \pi^* {\omega}$ are closed form on $\P(E)$.
In particular, near any point in $\P(E|_Z)$, they are $i \d\dbar$ of some smooth functions near this point.
Thus it is enough to modify the functions  $\tilde{\sigma}_* \tilde{\varphi}_\varepsilon+\varepsilon c \tilde{\sigma}_* \tilde{\pi}^* \psi_E$
such that it satisfies the condition in the theorem with a small loss of positivity.
By the theorem of Demailly, it is enough to consider the singular part of $\P(E|_Z)$.

Cover the singular part of $Z$ by open sets $U_{\alpha} \subset \C^{N_\alpha}$ such that over $U_\alpha$ the singular part is given by the zero sets of $(g_{\alpha,j})$.
Let $\theta_\alpha \in C_c^{\infty}(U_\alpha)$. 
Then 
$$f:=\sum_\alpha \theta_\alpha \log(\sum_j |g_{\alpha,j}|^2)$$ 
is a quasi-psh function over $Z$ with poles along the singular part of $Z$.
For any $\eta >0$,
$\tilde{\sigma}_* \tilde{\varphi}_\varepsilon+\varepsilon c \tilde{\sigma}_* \psi_E+\eta \pi^* f$
has pole along the singular part of $\P(E|_Z)$.
For any point in the singular part of $\P(E|_Z)$, the essential upper limit of the restriction of 
$\tilde{\sigma}_* \tilde{\varphi}_\varepsilon+\varepsilon c \tilde{\sigma}_* \psi_E+\eta \pi^* f$
to any irreducible component of the germ of $\P(E|_Z)$ at this point is infinite.
For $\eta$ small enough, we have
$$ \Theta(\cO_{\P(E)}(1)) + i \d \dbar (\tilde{\sigma}_* \tilde{\varphi}_\varepsilon+\varepsilon  c \tilde{\sigma}_* \tilde{\pi}^* \psi_E +\eta \pi^* f) \geq -2\varepsilon \pi^* {\omega} $$
which is a weakly quasi-psh function satisfying the condition in the theorem of Demailly, hence quasi-psh on $\P(E|_Z)$. 

The proof of the nef case is postponed to the end of the proof of Theorem 6.
Philosophically, as above, we can construct a quasi-psh function, which is smooth outside $\tilde{\sigma}(E)$.
Notice that by the embedded resolution of singularities of Hironaka, the image of the exceptional divisor under $\tilde{\sigma}$
can be taken to be contained in the singular part of $Z$.
Then we do an induction on dimension to construct a smooth quasi psh function on a neighbourhood of $\tilde{\sigma}(E)$.
By taking a regularised maximum of the above two quasi psh functions, we construct a smooth function on $\P(E|_Z)$.
\end{proof}
Now we extend the characterisation of nef line bundle of P\u{a}un to the higher rank case. 
Following the strategy of P\u{a}un, we first prove the following result.
The proof is very similar to the line bundle case.
\begin{mythm}
Let $E$ be a vector bundle over a compact K\"ahler manifold $(X, \omega)$.
Then $E$ is nef if and only if $E$ is strongly psef and for any irreducible subset $Z \subset X$ in the projection of non-nef locus of $\cO_{\P(E)}(1)$, 
for any $\varepsilon>0$, there exist a quasi-psh function $\varphi_\varepsilon \in C^\infty(\P(E|_Z))$ such that
$$\Theta(\cO_{\P(E)}(1)) + i \d \dbar \varphi_\varepsilon \geq - \varepsilon\pi^* \omega$$
on $\P(E|_Z)$ with $\pi: \P(E) \to X$.
\end{mythm}
\begin{proof}
The only if part is direct. We will focus on the other direction.

Let us first recall the following version of regularisation by Demailly \cite{Dem92}.

{\it
Let $T=\alpha+i \d \dbar \varphi$ be a closed positive $(1,1)-$current over a compact complex manifold $X$ where $\alpha$ is a smooth representative of the cohomology class.
Then for any $\varepsilon >0$, there exists a current $T_\varepsilon=\alpha+i \d \dbar \varphi_\varepsilon$ such that

(1) $\varphi_\varepsilon$ is smooth on $X\setminus E_{\varepsilon}(T)$ where $E_{\varepsilon}(T)$ is the closed analytic set on which the Lelong number of $T$ is larger than $\varepsilon$.

(2) In the sense of currents, we have
$$T_\varepsilon \geq -C\varepsilon \omega$$
where $C$ is a constant depending only on $X$, and $\omega$ is a fixed hermitian form on $X$.
} 

For any $\varepsilon>0$, let $\psi_\varepsilon$ be a quasi-psh function with analytic singularities over $\P(E)$ such that $\pi(\mathrm{Sing}(\psi_\varepsilon)) \neq X$ and
$$\Theta(\cO_{\P(E)}(1)) + i \d \dbar \psi_\varepsilon \geq - \varepsilon\pi^* \omega$$
in the sense of currents.
The existence of such functions is from the assumption that $E$ is strongly psef.
By the above regularisation theorem, for any $\varepsilon>0$, 
there exists $\tilde{\psi}_\varepsilon$  a quasi-psh function over $\P(E)$ such that 
it is smooth outside $E_\varepsilon(\psi_\varepsilon)=\{x \in \P(E), \nu(\psi_\varepsilon,x) \geq \varepsilon\}$ and
$$\Theta(\cO_{\P(E)}(1)) + i \d \dbar \tilde{\psi}_\varepsilon \geq - (C+1)\varepsilon\pi^* \omega$$
in the sense of currents.
Notice that $C$ is independent of $\varepsilon$.

By Grauert's direct image theorem, $\pi(E_{\varepsilon}(\psi_\varepsilon))$ is an analytic set of $X$.
Let $Z_{\varepsilon,k}$ be the global irreducible components of $\pi(E_{\varepsilon}(\psi_\varepsilon))$.
By assumption, for any $k$, 
there exist the prescribed smooth functions $\varphi_{\varepsilon,k}$ on $\P(E|_{Z_{\varepsilon,k}})$.
By Proposition 3.3 (1) in \cite{DP}, we can extend $\varphi_{\varepsilon,k}$ prescribed in the assumption as a smooth function $\tilde{\varphi}_{\varepsilon,k}$ on an open neighbourhood $V_{\varepsilon,k}$ of $\pi^{-1}(Z_{\varepsilon,k})$ 
satisfying the curvature condition
$$\Theta(\cO_{\P(E)}(1)) + i \d \dbar \tilde{\varphi}_{\varepsilon,k} \geq - 2\varepsilon\pi^* \omega.$$
For the convenience of the reader, we recall briefly the construction of the extension.
By adding $A \tilde{\omega}$ with large $A>0$ and $\tilde{\omega}$ a K\"ahler form on $\P(E)$, we are reduced to extending a smooth psh function over a closed analytic subset of $\P(E)$ to an open neighbourhood.
If the analytic set is smooth, we add large constant times the square of the
hermitian distance to the analytic set, which produces enough positivity in the normal direction.
When the analytic set is not smooth, we do an induction on the dimension of the strata (which are smooth manifolds) for the stratification of the analytic set.

We need the collaring lemma (e.g. Lemme de recollement 1.C.4 \cite{Pau98b})

{\it
Let $X$ be a compact complex manifold and $\omega$ a hermitian metric on $X$.
Let $A$ be an analytic subset of $X$ and $\cup_{k=1}^N Z_k$ be the global irreducible components of $A$.
Pour each $k$, assume that there exists a open neighbourhood $V_k$ of $Z_k$ in $X$ and a smooth real function $\varphi_k$ over $V_k$ such that
$i \d \dbar \varphi_k \geq \beta$
over $V_k$
where $\beta$ be a smooth form over $X$.

Then for any $\varepsilon >0$ there exists an open neighbourhood $\tilde{V}_\varepsilon$ of $A$ and a smooth real function $\tilde{\varphi}_\varepsilon$ over $\tilde{V}_\varepsilon$ such that
$$i \d \dbar \tilde{\varphi}_\varepsilon \geq \beta -\varepsilon \omega$$
over $\tilde{V}_\varepsilon$.
}

A similar construction can also be found in the proof of Proposition 3.3 (2) of \cite{DP}.
Applying the collaring lemma to $\tilde{\varphi}_{\varepsilon,k}$ gives a smooth function $\tilde{\varphi}_\varepsilon$ on an open neighbourhood of $\pi^{-1}(\pi(E_{\varepsilon}(\psi_\varepsilon)))$
satisfying
$$\Theta(\cO_{\P(E)}(1)) + i \d \dbar \tilde{\varphi}_\varepsilon \geq - 2\varepsilon\pi^* \omega-\varepsilon \tilde{\omega}$$
in the sense of currents.

Since $\cO_{\P(E)}(1)$ is relative $\pi-$ample, by interpolation, with abuse of notations, we have that 
$$\Theta(\cO_{\P(E)}(1)) + i \d \dbar \tilde{\varphi}_\varepsilon \geq -3 \varepsilon\pi^* \omega$$
in the sense of currents.

Let $\rho_\varepsilon$ be a quasi-psh function on $X$ such that $\rho_\varepsilon$ is smooth on $\P(E) \setminus \pi^{-1}(\pi(E_\varepsilon(\psi_\varepsilon))$ with logarithmic poles along $\pi^{-1}( \pi(E_\varepsilon(\psi_\varepsilon)))$ supported in the definition domain of $\tilde{\varphi}_\varepsilon$ and satisfies 
$$ i \d \dbar \rho_\varepsilon \geq -A \tilde{ \omega}$$
in the sense of currents for some constant $A$ depending on $\varepsilon$.
The construction of such a function can be initially found in Proposition 1.4 of \cite{Dem82}.
 We
denote by $\max_\varepsilon = \max  *\rho_\varepsilon$ a regularized max function and define for any $\delta>0$
$$\tilde{\rho}_{\varepsilon,\delta}:= \max_{\varepsilon} (\tilde{\psi}_\varepsilon+\delta   \rho_\varepsilon,\tilde{\varphi}_\varepsilon -C_\varepsilon).$$
$\tilde{\rho}_{\varepsilon,\delta}$ is smooth and coincide with $\tilde{\psi}_\varepsilon$ near the boundary of definition domain of $\tilde{\varphi}_\varepsilon$
for $C_\varepsilon>0$ large enough.
It satisfies the curvature condition
$$\Theta(\cO_{\P(E)}(1)) +i \d \dbar \tilde{\varphi}_\varepsilon \geq -\max(C+1,3) \varepsilon  \pi^*  \omega-\delta A \tilde{\omega}.$$
Choosing $\delta>0$ small enough and interpolating using the fact that $\cO_{\P(E)}(1)$ is relative $\pi-$ample finishes the construction of smooth potentials.
(Notice that the proof of collaring lemma uses exactly similar construction changing $\rho_\varepsilon$ to having a pole along the intersection of different irreducible components.)
\end{proof}
We need the following lemma to prepare for the induction on the dimension of $X$ in the following theorem 6.
\begin{mylem}
Let $E$ be a holomorphic vector bundle over an irreducible complex space $X$ of dimension 1.
Then $E$ is nef if and only if $E$ is strongly psef.
\end{mylem}
\begin{proof}
The fact that $E$ is nef implies that $E$ is strongly psef is trivial.
Let $\sigma: \tilde{X} \to X$ be the normalisation of $X$.
Since $\tilde{X}$ is of dimension 1, $\tilde{X}$ is smooth.
By definition, it is equivalent to showing that $\sigma^* E$ is nef if $\sigma^* E$ is strongly psef.
To simplify the notation, we assume $X$ is smooth in the following.

For any $\varepsilon >0$, there exists a quasi-psh function $\varphi_\varepsilon$ with analytic singularities over $\P(E)$ such that the projection of the singular part of $\varphi_\varepsilon$ is isolated points $p_{i,\varepsilon}$ and it satisfies the curvature condition
$$\Theta(\cO_{\P(E)}(1)) +i \d \dbar \varphi_\varepsilon \geq - \varepsilon \pi^* \omega$$
in the sense of currents.
Here $\omega$ is a K\"ahler form on $X$ and $\pi: \P(E) \to X$ is the projection.
$\Theta(\cO_{\P(E)}(1))$ is the Chern curvature of some smooth metric on $\cO_{\P(E)}(1)$.

To regularise the function $\varphi_\varepsilon$, we follow the arguments of lemma 6.3 \cite{Dem92} (with the same essential idea as in the proof of the above theorem).
Fix disjoint small coordinate neighbourhoods $V_{j,\varepsilon}$ of  $p_{i,\varepsilon}$.
Let $\theta_{j,\varepsilon}$ be
a cut-off function equal to 1 near $p_{i,\varepsilon}$ with support in $V_{j,\varepsilon}$. Then $\theta_{j,\varepsilon} \log |z - p_{j,\varepsilon} |$
has a Hessian which is bounded below by $-C(\varepsilon) \omega$ for some positive constant $C(\varepsilon)$.
 We
denote by $\max_\varepsilon = \max  *\rho_\varepsilon$ a regularized max function and define for any $\delta>0$
$$\tilde{\varphi}_{\varepsilon,\delta}:= \max_{\varepsilon} (\varphi_\varepsilon+\delta \pi^* \sum_j \theta_{j,\varepsilon} \log |z - p_{j,\varepsilon} |,-A)$$
with a large $A$.
Notice that $\varphi_\varepsilon$ is smooth outside $\pi^{-1}(p_{i,\varepsilon})$.
With a large choice of $A$ (depending on $\delta,\varepsilon$), $\tilde{\varphi}_{\varepsilon,\delta}$ is smooth and coincide with $\varphi_\varepsilon$ near the boundary of $\pi^{-1}(V_{j,\varepsilon})$.
It satisfies the curvature condition
$$\Theta(\cO_{\P(E)}(1)) +i \d \dbar \tilde{\varphi}_\varepsilon \geq -(\varepsilon +C(\varepsilon) \delta) \pi^*  \omega.$$
For $\delta$ small enough, we can assume that $C(\varepsilon) \delta <1$.
Applying the construction for any $\varepsilon>0$ implies that $E$ is nef.
\end{proof}
Now we are prepared for the proof of the following characterisation of nef vector bundles.
\begin{mythm}
Let $E$ be a vector bundle over a compact K\"ahler manifold $(X, \omega)$ of dimension $n$.
$E$ is nef if and only if for any irreducible analytic subset $Z \subset X$, $E|_Z$ is a strongly psef vector bundle.
\end{mythm}
\begin{proof}
By restricting the potentials prescribed in the definition of nef vector bundle along with the above lemma 8, the only if part is direct.

Let us focus on the inverse direction. We will do the induction on the dimension of $X$.
The case of dimension one is given in the above lemma 9.
Assume the result is valid for any compact K\"ahler manifold of dimension strictly less than $n$.
It is enough to check the condition in the above theorem 5.
By assumption, taking $Z=X$ implies that $E$ is strongly psef. 

Let $Z$ be an irreducible subset $Z \subset X$ in the projection of non-nef locus of $\cO_{\P(E)}(1)$.
Let $\tilde{\sigma}: \tilde{X} \to X$ be a modification of $X$ such that the strict transform $\tilde{Z}$ of $Z$ is a smooth submanifold of $\tilde{X}$.
Let $\sigma: \tilde{Z} \to Z$ be the restriction of $\tilde{\sigma}$.
By assumption, $\sigma^* E$ is a strongly psef vector bundle.
We shall show by induction that $\sigma^* E$ is in fact a nef vector bundle.

Let $\tilde{Y}$ be an irreducible analytic subset of $\tilde{Z}$.
By Grauert's direct image theorem, $Y:=\sigma(\tilde{Y})$ is an analytic subset of $Z$.
By assumption $E|_Y$ is strongly psef.
Then $\sigma^* E|_{\tilde{Y}}$ is strongly psef by Proposition 4. 
By the induction assumption, $\tilde{\sigma}^* E|_{\tilde{Z}}$ is in fact nef.

Since $\tilde{\sigma}^* E|_{\tilde{Z}}$ is nef, there exists for any $\varepsilon>0$, $\varphi_\varepsilon \in C^\infty(\P(\tilde{\sigma}^* E|_{\tilde{Z}}))$ such that
$$\tilde{\sigma}^* \Theta(\cO_{\P( E)}(1)) + i \d \dbar \varphi_\varepsilon \geq - \varepsilon \tilde{\pi}^* \tilde{\omega}$$
where $\Theta(\cO_{\P( E)}(1))$ is the Chern curvature of some smooth metric of $\cO_{\P( E)}(1)$, $\tilde{\omega}$ is K\"ahler and $\tilde{\pi}:\P(\tilde{\sigma}^* E)\to \tilde{X}$ is the projection.
We can extend $\varphi_\varepsilon$ to some open neighbourhood $V_\varepsilon$ of $\P(\tilde{\sigma}^* E|_{\tilde{Z}})$ with arbitrary small loss of positivity.
Since $\cO_{\P(\tilde{\sigma}^* E)}(1)$ is relatively ample, we can assume the lower bound is $ -2 \varepsilon \tilde{\pi}^* \tilde{\omega}$.
To simplify the notation, we also denote the extended smooth function as $\varphi_\varepsilon$.

Now we construct by induction on the dimension and from $\varphi_\varepsilon$ a smooth quasi-psh function demanded in the above theorem 5
which also finishes the proof of lemma 8.
The proof is essentially contained in the lemma 1.B.5 of  \cite{Pau98b}.
For the convenience of the reader, we sketch its proof.
The proof is an induction on the dimension of $Z$.
The dimension 1 case is proven in Lemma 9.

By the same process as above in lemma 8, we construct weakly quasi-psh functions  $\tilde{\sigma}_* \tilde{\varphi}_\varepsilon+\varepsilon c \tilde{\sigma}_* \psi_{\tilde{\pi}^{-1}\tilde{E}}$ near $\P(E|_Z)$ for some $c>0$ small enough.
Here $\tilde{E}$ is the exceptional divisor of $\tilde{\sigma}$.
They are smooth over the intersection of the open neighbourhood with the preimage of $\tilde{X} \setminus \tilde{E} \cong X \setminus \tilde{\sigma}(\tilde{E})$.
Let $Y_k$ be the irreducible components of the analytic set $\tilde{\sigma}(\tilde{E})$.
Now we construct smooth quasi-psh functions near $\pi^{-1}(Y_{k})$ satisfying the same curvature condition as $\tilde{\sigma}_* \tilde{\varphi}_\varepsilon+\varepsilon c \tilde{\sigma}_* \psi_{\tilde{\pi}^{-1}\tilde{E}}$ near each $\pi^{-1}(Y_{k})$.
With a small loss of positivity, we can use these smooth quasi-psh functions near $\pi^{-1}(Y_{k})$ and the smooth quasi-psh function $\tilde{\sigma}_* \tilde{\varphi}_\varepsilon+\varepsilon c \tilde{\sigma}_* \psi_{\tilde{\pi}^{-1}\tilde{E}}$ outside the preimage of $ \tilde{\sigma}(\tilde{E})$
to collar a smooth quasi-psh function over an open neighbourhood of $\pi^{-1}(Z)$ which would finish the proof.

Consider the restriction of $\tilde{\sigma}$ to the intersection $\tilde{\sigma}^{-1}(Y_k) \cap \tilde{Z} \to Y_k$.
Notice that this map is surjective since by the embedded resolution of singularities of Hironaka, the image of the exceptional divisor under $\tilde{\sigma}$ (containing $Y_k$)
is contained in the singular part of $Z$.

Consider the global irreducible components of $\tilde{\sigma}^{-1}(Y_k) \cap \tilde{Z}$.
We denote by $\tilde{Y}_k$ the irreducible component such that 
 the restriction of $\tilde{\sigma}$
is surjective.
Consider the restriction of $\tilde{\sigma}$  between the irreducible complex spaces $\tilde{Y}_k$ to $Y_k$.
The restriction of a nef vector bundle to an irreducible analytic subset is still nef.
Thus $\tilde{\sigma}^* E|_{\tilde{Y}_k}$ is a nef vector bundle.
By Proposition 4, $E|_{Y_k}$ is also a nef vector bundle.

By the dimension induction condition, we infer the existence of smooth quasi-psh functions near $\pi^{-1}(Y_k)$ satisfying the same curvature condition as $\tilde{\sigma}_* \tilde{\varphi}_\varepsilon+\varepsilon c \tilde{\sigma}_* \psi_{\tilde{\pi}^{-1}\tilde{E}}$ of each $Y_k$.
\end{proof}

Let us also observe that the main result of \cite{Wu20} can also be generalised to a compact manifold of the Fujiki class $\cC$. 
\begin{myprop}
Let $X$ be a compact manifold of the Fujiki class $\cC$.
Let $E$ be a holomorphic vector bundle over $X$.
Then the following conditions are equivalent:
\begin{enumerate}
\item $E$ is nef and $c_1(E)=0$;
\item $E$ is strongly psef and $c_1(E)=0$;
\item $E$ admits a filtration by vector bundles whose graded pieces are hermitian flat, i.e. given by unitary representations of $\pi_1(X)$.
\end{enumerate}
\end{myprop}
\begin{proof}
(1) implies (2) and (3) implies (1) trivially.
It is enough to show that (2) implies (3).

Let $\pi: \tilde{X} \to X$ a modification of $X$ such that $\tilde{X}$ is Kähler. $\pi^* E$ is strongly psef such that $c_1(\pi^* E)=\pi^* c_1(E)=0$ by assumption.
By the main result of \cite{Wu20}, $\pi^* E$ is numerically flat.
By Theorem 1.18 of \cite{DPS94},  $\pi^* E$ admits a filtration by vector bundles whose graded pieces are hermitian flat.
Notice that the K\"ahler condition is needed in the proof of Theorem 1.18 to apply the results of \cite{UY}.

Let $D$ be the exceptional divisor such that $\pi$ induces a biholomorphism between $\tilde{X} \setminus D$ and $X \setminus Z$.
Without loss of generality, we can assume that the codimension of $Z$ in $X$ is at least 2.
The restriction of $\pi^* E$ on $\tilde{X} \setminus D$ is equivalent to  unitary representations of $\pi_1(\tilde{X} \setminus D) \simeq \pi_1(X \setminus Z) \simeq \pi_1(X)$.
These define a vector bundle on $X$, which admits a filtration by vector bundles whose graded pieces are hermitian flat.
Since $E$ is reflexive, this vector bundle is isomorphic to $E$. 
\end{proof}
\begin{myrem}
{\em 
By the same argument of the proposition, we can also replace the vector bundle $E$ with some reflexive coherent sheaf $\cF$.
The only difference is using Lemma 10 of \cite{Wu20} to conclude that the first Chern class of the pullback of $\cF$ modulo the torsion part is 0.
Here the modification is chosen such that the pullback of $\cF$ modulo the torsion part is locally free and that the bimeromorphic model is K\"ahler.
Notice that in the paper \cite{Wu20}, Section 3 is under the K\"ahler assumption.
Since it uses only the H\"ormander's $L^2$ estimate, the same works for a general compact complex manifold.
}
\end{myrem}
As a geometric application, we have the following result.
\begin{myprop}
Let $X$ be a compact connected manifold of the Fujiki class $\cC$ such that $c_1(X)=0$.
Assume that there exists $q \neq n:=\mathrm{dim}_\C X$ such that $\wedge^q T_X$ is strongly psef.
Then up to a finite \'etale cover, $X$ is a complex torus.
In particular, an irreducible symplectic, or Calabi-Yau manifold of the Fujiki class $\cC$ does
not have strongly psef $\wedge^q T_X$ or $\wedge^q T^*_X$ for any $q \neq n:=\mathrm{dim}_\C X$.
\end{myprop}
\begin{proof}
By Proposition 5, $\wedge^q T_X$ is in fact nef.
By the result of \cite{LN19}, $\wedge^{n-1} T_X$ is also nef.
Since 
$$\Omega^1_X \simeq \wedge^{n-1} T_X \otimes K_X,$$
$\Omega^1_X$ is also nef.
Again by Proposition 5, $\Omega^1_X$ is numerically flat, which implies in particular that $T_X$ is nef.
By Proposition 3.6 of \cite{DPS94}, $X$ is K\"ahler as a consequence of regularisation.
Then it is well known that up to a finite \'etale cover, $X$ is a complex torus (cf. Corollary 9 of \cite{Wu20}).
The last statement is a consequence of the Beauville-Bogomolov decomposition theorem known in the compact K\"ahler case.
\end{proof}
Notice that the result does not hold without the Fujiki class condition.
For example, it is shown in \cite{DPS94}, that the tangent bundle of a Hopf surface is nef while its second Betti number is 0, which implies that its Chern class is 0 in the de Rham cohomology.

Notice that the definition of Schur products of modules is much more subtle. It seems to be a non-trivial question to ask whether the result of \cite{LN19} can be generalised to the category of coherent sheaves.

\section{Appendix. Surjectivity of Albanese morphism}
This appendix shows that the Albanese morphism of a compact complex manifold with nef anticanonical line bundle can be non-surjective. It can be shown by the following example due to Ueno \cite{Uen}.

In the general case, the Albanese morphism and the Albanese morphism for a compact complex manifold $X$ are defined as follows (cf. \cite{Uen74}).
For a compact complex manifold $X$, the Albanese torus is defined as
$$\Alb(X) :=H^0(X,d\cO_X)^*/H,$$
where $H^0(X,d\cO_X)\subset H^1(X,\C)$ is the space of closed holomorphic 1-forms and $H$ is the closure of the image of $H^1(X,\Z)$ in
$H^0(X,d\cO_X)^*$. Fixing a base point in $X$, integration over paths gives the Albanese map $\alpha: X \to \Alb(X)$, which is universal for pointed maps to complex tori.
\begin{myex}
{\em 
(Example 6.1 \cite{Uen})

Let $\pi: C \to \P^1$ be a double covering ramified at $2 g + 2$ points. Put
$L=\pi^*\cO(1), F=L^n$.
For any point $t \in \Pic^0(C)$,
we denote
$\cO(t)$ the corresponding line bundle of degree 0 on $C$. Put $F_t  = F \otimes \cO(t)$.
There
exists a line bundle $\cF$ on $C \times \Pic^0(C)$
such that $\cF|_{C \times \{t\}} =F_t$.
Assume $n \geq g$. Then
$F_t$ is globally generated.
Let $ p : C \times \Pic^0(C) \to \Pic^0(C)$ be the natural projection.
Since $p_* \cF$ is a vector bundle (by Grauert's direct image theorem),  there exists an open neighbourhood $U$ of the origin of $\Pic^0(C)$
and two
holomorphic sections $\varphi, \psi$ of $\cF$ over $p^{-1}(U)$ such that for each $t \in U$,
$\varphi_t:=\varphi|_{C \times \{t\}}, \psi_t:=\psi|_{C \times \{t\}},$
regarded as elements of $H^0(C \times \{t\},
F_t)$, have no common zero.
Define
$$I_1=\begin{pmatrix} 
1 & 0 \\
0 & 1 
\end{pmatrix},
I_2=\begin{pmatrix} 
0 & 1 \\
-1 & 0 
\end{pmatrix},
I_3=\begin{pmatrix} 
0 & i \\
i & 0 
\end{pmatrix},
I_4=\begin{pmatrix} 
i & 0 \\
0 & -i 
\end{pmatrix}.
$$
Then, for any point $t \in U$,
$\Lambda_t:= \sum_{2 \leq j \leq 4} \Z I_j (\varphi_t, \psi_t)^t$ is a lattice of each fibre of a vector
bundle $V_t = F_t \oplus F_t$ over $C$. The group $\Lambda_t$ acts on $V_t$ as translations in each fibre. Define
$M_t$ to be the quotient manifold and $\pi_t:M_t \to C$ to be the natural morphism.
Let $f:\tilde{C} \to C$  be a double covering ramified at the divisor $(s)$ where $s$ is a generic
element of $H^0(C,
L^{2k})$.
Define $\tilde{M}_t=M_t \times_C \tilde{C}$ and $\tilde{\pi}_t: \tilde{M}_t \to \tilde{C}$ be base change.
Then we have
$$K_{\tilde{M}_t}=\tilde{\pi}_t^* f^* (L^{g-1+k-2n} \otimes \cO(2t)).$$
Then for $k \leq 2n-g+1$, $-K_{\tilde{M}_t}$ is nef as pull back of a nef line bundle.
(Notice that
for $k = 2n-g+1$, $\kappa(\tilde{M}_t)=-\infty$ if $2t$ is non-torsion and $\kappa(\tilde{M}_t)=0$ if $2t$ is torsion.
On the other hand, since $K_{\tilde{M}_t}$ is nef in this case, the numerical dimension $\mathrm{nd}(K_{\tilde{M}_t}):= \max\{p\in \N, c_1(K_{\tilde{M}_t})^p \in H^{p,p}_{BC}(X, \R) \neq 0\}$ is 0.
In fact, $c_1(K_{\tilde{M}_t})=0 \in H^{1,1}_{BC}(X, \R)$.
It provides a counterexample of the generalised abundance conjecture, that is, the dimension of $K_{\tilde{M}_t}$ is equal to its Kodaira dimension,  cited in \cite{BDPP} in the non-K\"ahler case.)

For any $k$ and $t \in U$, we have
$$H^0( \tilde{M}_t, \Omega^1_{\tilde{M}_t}) =\tilde{\pi}_t^*  H^0( \tilde{C}, \Omega^1_{\tilde{C}}) .$$
Hence, the Albanese torus of $\tilde{M}_t$ is isomorphic to the Jacobian variety $J(\tilde{C})$ of the
curve $\tilde{C}$ and the composition $j \circ \tilde{\pi}_t : \tilde{M}_t \to \tilde{C}  \to J(\tilde{C})$ is the Albanese mapping where
$j: \tilde{C}  \to J(\tilde{C})$ is the Albanese mapping of $\tilde{C}$.
In particular, the Albanese morphism of $\tilde{C}  \to J(\tilde{C})$ is not surjective.
Thus the Albanese morphism of $\tilde{M}_t$ is also not surjective.

Notice that for $k \leq 2n-g+1$, the Kodaira dimension of $\tilde{M}_t$ is either 0 or $-\infty$ which corresponds to the case when $-K_{\tilde{M}_t}$ is nef.
By the Main Theorems of \cite{Uen}, $\tilde{M}_t$ is not of class $\cC$ in these cases.

We hope in general that the Albanese morphism of a compact manifold of class $\cC$ with an anticanonical line bundle should always be surjective.
}
\end{myex}
\begin{myrem}
{\em 
With the same notations as in the above example,
$\tilde{M}_t$ admits a fibration onto $\tilde{C}$ with 2 dimensional torus as fibers.
By homotopy exact sequence, the natural morphism
$\tilde{\pi}_{t,*}: \pi_1(\tilde{M}_t) \to \pi_1(\tilde{C})$ is surjective.
In particular, $\pi_1(\tilde{M}_t)$ is not almost abelian.
This shows that in general, for a compact complex manifold $X$ with $-K_X$ nef, the fundamental group of $X$ is not necessarily almost abelian.
Notice that by the work of P\u{a}un \cite{Pau98b} (cf. also \cite{DPS93}), this is always the case if $X$ is K\"ahler.
It is natural to conjecture that for a compact complex manifold $X$ of the Fujiki class $\cC$ with $-K_X$ nef, the fundamental group of $X$ is almost abelian.
}
\end{myrem}
  
\end{document}